\documentclass[a4paper,twoside,11pt]{article}
\usepackage[a4paper,left=2cm,right=2cm,top=1cm,bottom=1cm]{geometry}
\usepackage{a4wide}
\usepackage[utf8]{inputenc}
\usepackage[english]{babel}
\usepackage{color} 
\usepackage{amsmath,amscd,amssymb,amsthm,amsfonts}
\usepackage{xparse}
\usepackage{mathabx}
\usepackage{mathrsfs}

\usepackage{bbm}
\usepackage{graphicx}
\usepackage{hyperref}
\usepackage{enumitem}

\providecommand{\keywords}[1]
{
  \small	
  \textbf{\textbf{Keywords.}} #1
} 

\newcommand{\ddr}{\mathrm{d}}

\theoremstyle{plain}
\theoremstyle{plain}
\newtheorem{theorem}{Theorem}[section]
\newtheorem{thA}{Theorem}

 \newtheorem{lemma}[theorem]{Lemma}
 \newtheorem{example}{Example}[section]
 \theoremstyle{definition}
 
 \newtheorem{proposition}{Proposition}[section]
 \theoremstyle{remark}
 
 \makeatletter

\usepackage{color}

\usepackage{import}

 \title{On the coming down from infinity of continuous-state branching processes with drift-interaction}

 \author{Félix Rebotier\footnote{CMAP, Ecole Polytechnique. Email: felix.rebotier@polytechnique.edu}}

\begin{document}
\maketitle
\begin{abstract} 
We study the phenomenon of coming down from infinity – that is, when the process
starts from infinity and never returns to it – for continuous-state branching processes with
generalized drift. We provide sufficient conditions on the drift term and the branching
mechanism to ensure both non-explosion and coming down from infinity, without requir-
ing the associated jump measure to have a finite first moment. Assuming the process comes down from infinity and the drift satisfies a one-sided Lipschitz condition, we show that, as the initial values tend to infinity, the process converges locally uniformly almost surely to the strong solution of a stochastic differential equation. The main techniques employed are comparison principles for solutions of stochastic equations and the method of Lyapunov functions, the latter being briefly reviewed in a broader setting.
%The extinction of the process is also studied and 
% sufficient conditions for having extinction of the process or escape to $\infty$ are designed.
\end{abstract}

\section{Introduction}
Continuous-state branching processes (CBs) form a well-established class of non-negative Markov processes with no negative jumps; we refer for instance to the books of Kyprianou \cite[Chapter 12]{Kyprianou} and of Li~\cite[Chapter 3]{ZLi}. They model random continuous population in which individuals reproduce and die independently according to a common law. Over the past two decades, the CB processes have been generalized in various directions to incorporate for instance some phenomena of interaction. Notably, Lambert, in a seminal work \cite{Lambert}, has defined the class of CB processes with logistic growth.  In this framework, ``pairwise competition" among individuals is modeled by adding a negative quadratic drift term to the dynamics. This competition term has been then generalized to other functions allowing for more general competition pressure (with negative or positive density-dependence) on the population according to its size, see for instance Berestycki et al. \cite{Berestycki}, Le and Pardoux \cite{LePardoux}, Li and al. \cite{LWZqsd}, Pardoux \cite{Pardouxbook}  and Dram\'e and Pardoux \cite{DramePardoux}. Let us also mention the recent work of Li et al. \cite{LYZ}, where a more general class, termed nonlinear branching processes, is introduced. 
%A common feature in the works \cite{Berestycki, LePardoux,LYZ} lies on a first moment assumption on the L\'evy measure.  
\smallskip

In this article, we focus on CB processes with a generalized drift.   The latter arise as scaling limits of Bienaymé-Galton-Watson processes with interaction, see \cite{DramePardoux} and Bansaye et al. \cite{BansayeCaballeroMeleard}. We introduce the term continuous-state branching processes with \textit{drift-interaction}, abbreviated as $\mathrm{CBDIs}$, to refer to this class of models. More precisely, we refer to any process that is a solution to a stochastic equation of the following form as a $\mathrm{CBDI}$ process:
\begin{equation}
\label{eds}
\begin{split}
X_t = X_0 + \sigma \int_0^t \sqrt{X_s} \, \ddr B_s &- \gamma \int_0^t X_s \, \ddr s + \int_0^t \int_0^{X_{s-}} \int_{(0,1]} h \, \tilde{M}(\ddr s,\ddr u,\ddr h) \\
&+ \int_0^t \int_0^{X_{s-}} \int_{(1,\infty]} h \, M(\ddr s,\ddr u,\ddr h) - \int_0^t I(X_s) \, \ddr s, \quad t < \zeta,
\end{split}
\end{equation}
where $X_0$ is a positive random variable, $\sigma\geq 0$, $\gamma\in \mathbb{R}$ and $M$ is a Poisson random measure (PRM) on $\mathbb{R}_+\times \mathbb{R}_+\times [0,\infty]$ with intensity $m(\ddr s,\ddr u,\ddr h):=\ddr s\otimes \ddr u\otimes\pi(\ddr h)$, $\pi$ being a L\'evy measure, with possibly a mass at $\infty$, satisfying $\int_0^\infty 1\wedge h^2\pi(\ddr h)<\infty$. The random signed measure $\tilde{M}$ denotes the compensated PRM, that is $\tilde{M} := M - m$. The function $I:\mathbb{R}_+\to \mathbb{R}$ is an additional general drift term. Finally, the random variable $\zeta$ in \eqref{eds}, referred to as the lifetime of $X$, is defined as the first exit time from the interval $(0,\infty)$ of the process $X$. 
 % We shall return later to the conditions asked on $I$ and 
%We explain 
%later conditions smooth enough for having a strong unique solution to \eqref{eds} up to time $\zeta$.

\medskip
The stochastic integrals and drift terms in \eqref{eds} have the following heuristic interpretation.
\begin{itemize}
\item[(i)] The reproduction is governed by a continuous and a jump term: 
\begin{itemize}
    \item The continuous terms include a linear drift $-\gamma \int_0^t X_s\, \ddr s$, representing a deterministic exponential growth or decay, and the classical Feller diffusion term $\sigma \int_0^t \sqrt{X_s}\, \ddr B_s$, which arises for instance as the scaling limit of discrete binary branching.  Alternatively\footnote{This will play a role later in the article as it allows one to construct  $\mathrm{CBDIs}$ for all initial values on the same probability space.}, following Dawson and Li \cite{DawsonLi}, for a given time-space Gaussian white noise $W$ on $\mathbb{R}_+ \times \mathbb{R}_+$ with intensity $\ddr s \ddr u$, there exists a Brownian motion $B$ such that  
 \begin{equation}\label{whitenoise}
\sigma \int_0^t \int_0^{X_{s-}} W(\ddr s,\ddr u)=\sigma \int_0^t \sqrt{X_{s-}}\, \ddr B_s.
 \end{equation}
Loosely speaking, this formulation reflects the idea that, continuously in time, each individual in the population contributes an infinitesimal positive or negative random mass, modeling random fluctuations at the individual level.
\item The jump terms model large reproduction events. At an atom of time $s$ of $\mathcal{M}$, an individual $u$ is chosen uniformly at random in $[0,X_{s-}]$ and reproduces according to the jump measure $\pi$. The integral on $(0,1]$ with respect to $\tilde{M}$, represents small reproduction events counterbalanced by bursts of negative drift, which loosely speaking can be thought as ``natural death". The integral over $(1,\infty]$, with  no compensation (finite variation parts) corresponds to large reproduction events.
\end{itemize}

\item[(ii)] The drift term $-\int_0^t I(X_s)\,\ddr s$ encodes interaction effects such as \textit{competition} or \textit{cooperation}. When $I$ is differentiable, this term can be rewritten as
\[
-\int_0^t I(X_s)\, \ddr s = -\int_0^t \int_0^{X_s} I'(u)\, \ddr u\, \ddr s - I(0)t
\]
and we may interpret it as an additional deterministic effect at the level of individuals. Each individual $u$ in the population $[0,X_{s}]$ contributes deterministically to the total drift through a rate $I'(u)\,\ddr u$, which either slows down or accelerates the population growth. Loosely speaking, a positive derivative $I'(u) > 0$ corresponds to competition, while a negative derivative reflects cooperative effects. Thus, for a given ``individual" $u$, there is an additional infinitesimal rate $I'(u)\ddr u$, either of death or of birth, depending on the sign of $I'(u)$. The term $-I(0)$, when positive, can be seen as a constant immigration rate. 
\end{itemize}

Denote by $\Psi$ the Lévy-Khintchine function associated to the triplet ($\sigma$,$\gamma$, $\pi$):
\begin{equation}
\label{eq LK}
\Psi(z) = -\lambda + \frac{\sigma^2}{2} z^2 + \gamma z + \int_{(0,\infty)} \left( e^{-zh} - 1 + zh \mathbf{1}_{[0,1]}(h) \right) \pi(\ddr h), \qquad z\in [0,\infty).
\end{equation}
 It governs the branching dynamics and is referred to as the branching mechanism.  The parameter $\lambda\geq 0$  stands for a killing term. It can be encompassed in the L\'evy measure $\pi$ as a mass at $\infty$, $\pi(\{\infty\})=\lambda$, and interpreted as the rate of a jump to $\infty$. \vspace{2mm}
 
 When $I \equiv 0$, the solution to the stochastic equation \eqref{eds} is a $\mathrm{CB}$ process with mechanism $\Psi$, see e.g. \cite{ZLi,JiLi}. In particular, the process $X$ satisfies the branching property, namely, for all $x,y\in [0,\infty]$, $\mathbb{P}_{x+y}=\mathbb{P}_x\star \mathbb{P}_y$, where $\mathbb{P}_x$ denotes the law of the process $X$ with $X_0=x$, and both boundaries $\infty$ and $0$ are absorbing. 
 %When $I$ is constant equals to $\delta\in (0,\infty)$, the process solution to \eqref{eds} - if we do not stop it at $0$ - is a continuous-state branching process with immigration (CBI), denoted by CBI($\Psi$, $\Phi$), with immigration mechanism $\Phi := \delta \operatorname{Id}$. See \cite[Chapter 3 Section 4]{ZLi} for a general study of CBI processes. 
 %We will analyse the case with a linear immigration mechanism in \ref{sec:extinctionproof}.
 %\vspace{2mm}
 
The primary object of interest in this article is the case when the function $I$ is not linear. In this setting, the process $X$, solution to \eqref{eds}, is not a CB process, as in particular it does not satisfy the branching property. %We call it CB process with branching mechanism $\Psi$ and drift-interaction mechanism $I$, abbreviated in CBDI($\Psi$, $I$). %In this article, we study the behavior of CBDIs near the boundary points $0$ and $\infty$. 
% The phenomena of explosion (respectively extinction), namely the fact that the process can reach $\infty$ (respectively $0$) in finite time with positive probability are completely understood in the branching setting; i.e. when $I\equiv 0$, we refer e.g. to Kyprianou's book \cite{Kyprianou}. 
% These phenomena are well understood in the CB case.
% The CB process does not explode almost surely if and only if 
% \begin{equation}
% \label{eq dynkin}
%     \int_0 \frac{du}{\mid\Psi(u)\mid} = \infty \qquad \text{(Dynkin's condition).}
% \end{equation}
% It does not go extinct almost surely if and only if 
% \begin{equation}
%     \label{eq grey}
%     \int^\infty \frac{du}{\Psi(u)} = \infty \qquad \text{(Grey's condition).}
% \end{equation}
% We refer for instance to Kyprianou's book \cite[Theorem 12.3 p. 341 and Th 12.5 p. 343]{Kyprianou}. 
The drift term $I$ leads to 
richer behaviors at the boundaries  than those of CBs. In particular, when the drift compensates the large jumps sufficiently, the process may start from infinity, meaning that the process started from $X_0 = \infty$ -- a proper definition is given in the forthcoming Section \ref{section:lyapfun} -- becomes almost surely finite at any strictly positive time. The latter phenomenon is also called  in the literature \textit{coming down from infinity}. 

%In a similar fashion, the drift term $I$ can prevent an indefinite growth of the process and force the latter to get extinct despite the linear drift being positive ($\gamma<0$).
%This phenomenon cannot occur for a pure CB, where $\infty$ is an absorbing boundary.

The logistic branching process, defined in \cite{Lambert}, corresponds to the specific case where the function $I$ is given by $x\mapsto \frac{c}{2}x^2$. A complete classification of the boundaries has been carried out for this model in  Foucart \cite{FEJP}.  Numerous works have addressed extinction, explosion, and entrance-type behaviors for other generalized CB processes. We refer for closely related studies to \cite{LePardoux}, Le \cite{VLe}, Palau and Pardo \cite{PalauPardo}, Leman and Pardo \cite{LemanPardo}, and \cite{LYZ}.

To the best of our knowledge, with the exception of the logistic case, most studies in the literature assume that the L\'evy measure satisfies certain moment conditions, typically a finite first or log moment. In this article, we initiate the study of the boundary behavior of CBDI processes under relaxed moment conditions. This serves as a foundation for a broader investigation - incorporating the possibility of regular boundaries - which will be addressed in a forthcoming work.

We first establish explicit sufficient conditions for non-explosion and coming down from infinity. This completes previous results obtained in \cite{LYZ} and \cite{LePardoux} by allowing the branching dynamics to be driven by Lévy measures with infinite mean.  Our approach is based on comparison properties of stochastic equations and classical Lyapunov-type methods, which differs from those employed in the aforementioned works, and will be reexplained in a broad context. 

The second contribution of this article is the analysis of the stochastic equation solved by the process starting from infinity. We will adapt techniques from the works of Li \cite{PLi}, Bansaye \cite{Bansaye}, and Bansaye et al. \cite{BansayeColletMartinezMeleard}, which deal with related but distinct models.
% Under the assumption of coming down from infinity, we identify the stochastic differential equation satisfied by the process starting from $\infty$. 

The structure of the paper is as follows. Section~\ref{sec:prelminaries} gathers some fundamental facts about existence and uniqueness of the solution to the stochastic equation \eqref{eds}, as well as comparison theorems, the Markov property and the identification of its infinitesimal generator. Section~\ref{section:main} presents the main results. Before providing proofs, we give in the autonomous Section \ref{sec:lyapunov}, general background on the method of Lyapunov functions, see Theorem \ref{th cond lyapunov}. The proofs of the main results are gathered in Section~\ref{section:proofs}. Section~\ref{section:lyapfun} introduces some Lyapunov functions for $\mathrm{CBDIs}$ and applies Theorem \ref{th cond lyapunov} to establish Theorem~\ref{th descente inf}. Section~\ref{section:regularity} investigates the regularity of the process with respect to the initial condition. Section~\ref{section:edsinf} adresses the stochastic differential equation satisfied by the process coming down from infinity. 

\section{Preliminaries}\label{sec:prelminaries}
\textbf{Notations.} For any real function $f$, we denote by $D_f$ its domain of definition. We also introduce the following spaces of functions: $C^1$ (resp. $C^2$) is the continuously (resp. twice) differentiable functions on their domain. The space \(C_b^2\) denotes the set of twice continuously differentiable functions that are bounded.

%The spaces $C_0$ and $C^2_b$ are the  continuous functions, which are  vanishing at $\infty$ and bounded, respectively.  %The subspace of $C^2$ functions defined on (0,$\infty)$ with compact support is denoted by $C^2_c((0,\infty))$. 

 We denote by $[0,\infty]$  the extended half-line.  Any integral term written as $\int_s^{t}\dots$ is meant to be $\int_{(s,t]}\dots$ for $s\in [0,\infty)$ and $t\in (0,\infty]$. For any measure $\pi$  on the Borelian sets of $[0,\infty]$, we denote its tail  by $\overline{\pi}(u):=\pi([u,\infty])$ for all $u>0$. 
 
 Set $e^{-\infty}:=0$ and endow $[0,\infty]$ with the compact metric :
$$d(x,y)=|e^{-x}-e^{-y}|, \quad (x,y)\in [0,\infty].$$
The Skorokhod space of $[0,\infty]$-càdlàg functions, is denoted by $D$, carrying the uniform distance $\rho_\infty$:  
\begin{equation}
    \label{def rhoinf} \forall u,v \in D, \quad \rho_{\infty}(u,v)=\underset{t\geq 0}{\sup} \ d(u(t),v(t)).
\end{equation}
For any positive càdlàg process $X$ and $x\in [0,\infty)$, we denote by $X^x$ the process with initial state $x$. For $a \in (0 ,\infty]$, $b \in [0,\infty)$ and $x\in  [0,\infty)$, we define the first  passage times below level $a$ and above $b$ as follows:

\begin{equation*}
\tau_a^{x,-}=\inf\{t\geq 0, \ X^x_t\leq a\}, \quad
\tau_b^{x,+}=\inf\{t\geq 0, \ X^x_t\geq b\}.
\end{equation*}
We extend this definition to $a=0$ and $b=\infty$, by taking the a.s. limits of the previous random times:
\begin{equation*}
\tau_0^{x,-}=\underset{a\rightarrow 0}{\lim} \  \tau_a^{x,-}, \quad
\tau_{\infty}^{x,+}=\underset{b\rightarrow \infty}{\lim} \ \tau_b^{x,+}.
\end{equation*}
These definitions are valid since the maps $[0,x]\ni a\mapsto \tau_a^{x,-}$ and $[x,\infty)\ni b\mapsto \tau_b^{x,+}$ are respectively non-increasing and non-decreasing a.s.. To alleviate the notations, we will often use $\tau_a^{-/+}$ instead of $\tau_a^{x,-/+}$ and $X$ instead of $X^x$. We call $\tau_0^{-}$ the extinction time of $X$ and $\tau_{\infty}^{+}$ the explosion time of $X$.

% Recall the stochastic differential equation \eqref{eds}. The drift interaction term is governed by $-I(X_t)\ddr t$. We first give condition on $I$ such that the stochastic equation \eqref{eds} admits a unique strong solution. This is the basic framework to define the CBDI process. 
% These conditions, which we denote by [A], are standard assumptions. The existence and uniqueness result for the CBDI is stated in Proposition \ref{Prop existence}. Additionally, the set of conditions [A] guarantees important comparison properties for the solutions of \eqref{eds}, stated in Proposition \ref{prop comp}. As expected, one such property concerns the initial state : the population size, i.e. the CBDI($\Psi$,$I$) almost surely increases with the number of individuals at time $t=0$. Another important comparison property concerns the interaction function $I$. Specifically, if two functions $I^{(1)}$ and $I^{(2)}$ satisfy [A] and $I^{(1)}\leq I^{(2)}$, then the corresponding CBDI processes $X^{(1)},X^{(2)}$ satisfy $X^{(1)} \geq X^{(2)}$ a.s..  
\medskip

We 
%do not seek in this article the most general setting in order that the stochastic equation \eqref{eds} admits a unique strong solution, and 
focus in this article on the setting where the function $I:\mathbb{R}_{+}\mapsto \mathbb{R}$  satisfies
%\begin{equation*}
\paragraph{Condition [A]}: \begin{center} $I(0) \leq 0 \text{ and } I \text{ is locally Lipschitz on } (0,\infty)$.\end{center}
%\begin{center}
% Condition [A]: \quad \qquad $I(0) \leq 0 \text{ and } I \text{ is locally Lipschitz on } (0,\infty)$.
% \end{center}

%As we shall see below, this is sufficient for existence of a strong unique solution to \eqref{eds}. 
%\end{equation*}
% \paragraph{Condition [A]}
% %\textit{For the function $I$, we denote the set of conditions [A] :}
% \begin{itemize}
% \item[(A1)] $I(0) \leq 0$.
% \item[(A2)] $I$ \textit{is locally Lipschitz on} $(0,\infty)$.
% \item[(A3)] $\hat{\pi}(1)>0$. (je sais pas encore si je mets cette hypothèse ici ou plus haut) Mets la plus haut, avant we introduce our basic assumption. Qu'est ce que c'est $\hat{\pi}$? tu as défini $\bar{\pi}$? We focus in this article on the setting with jumps and assume from now on that $\bar{\pi}(1)>0$. 
% \end{itemize}
%This condition ensures existence of a unique strong solution to \eqref{eds}. Additionally, it guarantees important comparison properties for the solutions of \eqref{eds} with respect to the drift functions and to the initial values. Those are standard assumptions and details for all these facts are provided to the forthcoming Proposition \ref{prop comp}. 
Suppose that $\big(\Omega,\mathcal{F},(\mathcal{F}_t)_{t\geq 0},\mathbb{P}\big)$ is a filtered probability space satisfying the usual hypotheses. Recall $\mathcal{M}$ and $B$ the Poisson random measure and the Brownian motion in the stochastic equation \eqref{eds} and assume that they are adapted to $(\mathcal{F}_t)_{t\geq 0}$. 
\begin{proposition}
\label{Prop existence}
Assume that $I$ satisfies [A]. Then, for all positive $\mathcal{F}_0$-measurable random variable $X_0$, there exists a unique càdlàg strong solution to $\eqref{eds}$, up to the random time $\zeta$.
\end{proposition}
\begin{proof}
This follows from general results on stochastic differential equations with jumps. We do not go into details here since existence and uniqueness results are available in the literature for more general settings. We refer  for instance to Li and Pu \cite[Theorem 6.1]{li-pu} and \cite[Theorem 2.5, p.~823]{DawsonLi}. It is worth mentioning however that some of the assumptions in these references (in particular conditions (2a) and (2c) in \cite{DawsonLi}) require the jump intensity measure $\pi$ to have a finite first moment, i.e. $ \int^{\infty}_1h\pi(\ddr h)<\infty$. As noticed in Palau and Pardo \cite[Proposition 1, p.~61]{PalauPardo}, \cite[p.~2539]{LYZ}, see also Ji and Li \cite{JiLi}, a way to circumvent this assumption is to proceed by localization and truncation of the large jumps.
%: jumps larger than an integer $n$ are replaced by the value $n$. 
%In addition, since the CBDI process is well-defined up to its first exit time from $(0,\infty)$, we can apply a localization argument. 
All conditions are indeed then verified on compact intervals of the form $[\frac{1}{m}, m]\subset(0,\infty)$. 
%The local Lipschitz continuity assumption on the function $I$ ensures also that its growth is at most linear on compact subsets of $(0,\infty)$. 
The CBDI process thus exists uniquely  up to its first exit time from $(0,\infty)$. 
\end{proof}

The first process we shall work with is the minimal solution of \eqref{eds}, that is to say we consider the process that gets absorbed at the endpoints $0$ and $\infty$ whenever $\zeta=\tau_0^{-}\wedge \tau_{\infty}^{+}<\infty$. More precisely, we set
\begin{equation}
X_t:= \left\{
\begin{array}{ll}
0 & \text{if } t \geq \tau_0^{-}, \quad \text{on } \{\tau_0^{-} < \tau_{\infty}^{+} \}, \\
\infty & \text{if } t \geq \tau_{\infty}^{+}, \quad \text{on } \{\tau_{\infty}^{+} < \tau_0^{-} \}.
\end{array}
\right.
\end{equation}
For any $x\in [0,\infty)$, we denote by $(X^x_t, t\geq 0)$  the process solution to \eqref{eds} with $X_0=x$. 
% A fundamental tool we are going to use throughout this article is comparison.  
% % {\color{red}  Voir Peisen pour les mesures $\mathbb{P}$ et $\mathbb{P}_x$.} %Under condition [A], one can define the family of solutions to \eqref{eds}, indexed by their initial values $(X_t(x),t\geq 0, x\geq 0)$ on the same probability space, see for instance Berestycki et al. \cite{Berestycki} for a similar setting as ours. }
% %We refer once again to Dawson and Li's work, specifically \cite[Theorem 2.2]{DawsonLi}.
% \begin{proposition}\label{prop comp}\ Assume that $I$ satisfies [A]. Then the following comparison properties hold:
% \begin{enumerate}
%     \item Let $x,y\in [0,\infty)$, \begin{equation}
% \textup{(CP1)} \quad x \leq y \ \Rightarrow \ \mathbb{P}\big( X_t^x\leq X_t^y, \ \forall t\in [0,\zeta(x)\wedge\zeta(y))\big)=1
% \end{equation}
% \item Let $x\in [0,\infty)$ and  $I^{(1)}$ and $I^{(2)}$ be two functions defined on $[0,\infty)$, satisfying [A] and $X^{(1)}$, $X^{(2)}$ be the CBDI($\Psi$, $I^{(1)}$) and CBDI($\Psi$, $I^{(2)}$) started from $x$. We have :
% \begin{equation}
% \textup{(CP2)} \quad I^{(1)} \leq I^{(2)} \ \Rightarrow \ \mathbb{P}\Big(X^{(2)}_t \leq X^{(1)}_t,\ \forall t\in [0,\zeta^{(1)}\wedge\zeta^{(2)})\Big)=1.
% \end{equation}
% \end{enumerate}
% \end{proposition}
% We refer the reader to \cite[Theorem 2.2]{DawsonLi} for those statements.\\
Its law is denoted by $\mathbb{P}_x$ and the expectation with respect to  $\mathbb{P}_x$ by $\mathbb{E}_x$. 
\smallskip

We will also rely on the Markovian properties of $\mathrm{CBDI}$ processes in our analysis. We call \textit{generator} of $X$ the operator
\begin{align}
\label{generator}
\mathcal{X}f(x):=-I(x)f'(x)&-\gamma xf'(x)+\frac{\sigma^2}{2}xf''(x) \nonumber \\
&+x\int_0^{\infty}\left(f(x+u)-f(x)-uf'(x)\mathbf{1}_{\{u\leq 1\}}\right)\pi(\ddr u), \ \ x\in D_f,
\end{align}
defined on $D_\mathcal{X}:=\left\{f\in C^2: \mathcal{X}f(x) \text{ is well-defined for all } x\in D_f\right\}$. 
%We call $\mathcal{X}$ the extended generator of the process $X$. 
\begin{proposition}
\label{lemma generator} 
The process $X$, solution to \eqref{eds}, satisfies the strong Markov property and the following local martingale problem: 
\begin{equation}\label{eq:localmartingaleproblem}
\forall x\in [0,\infty),\  \forall f \in \mathcal{D}_{\mathcal{X}}, \quad \left(f(X_t)-\int_0^t\mathcal{X}f(X_s)\ddr s\right)_{t\geq 0} \text{ is a } \mathbb{P}_x\text{ - local martingale.}
\end{equation}
Furthermore, $C^2_b\subset D_\mathcal{X}$ and for any $f\in C^2_b$ such that $\underset{z>0}{\sup} \ \vert \mathcal{X}f(z) \vert<\infty$, the process in \eqref{eq:localmartingaleproblem} is a martingale. 
\end{proposition}
% \noindent In particular, note that for any $f\in C^2_b([0,\infty))$,  $\mathcal{X}f$ is well-defined and when furthermore, one has $\underset{z>0}{\sup} \ \vert \mathcal{X}f(z) \vert<\infty$, then the process in \eqref{eq:localmartingaleproblem} is a martingale. 
The proof of Proposition \ref{lemma generator} is deffered to the Appendix.
\\ 

A fundamental tool we are going to use throughout this article is comparison.  We refer the reader to \cite[Theorem 2.2]{DawsonLi} for those statements\footnote{As in the proof of Proposition \ref{Prop existence}, they follow from \cite[Theorem 2.2]{DawsonLi} by truncation and localization.}.
% {\color{red}  Voir Peisen pour les mesures $\mathbb{P}$ et $\mathbb{P}_x$.} %Under condition [A], one can define the family of solutions to \eqref{eds}, indexed by their initial values $(X_t(x),t\geq 0, x\geq 0)$ on the same probability space, see for instance Berestycki et al. \cite{Berestycki} for a similar setting as ours. }
%We refer once again to Dawson and Li's work, specifically \cite[Theorem 2.2]{DawsonLi}.

\begin{proposition}\label{prop comp}\ Assume that $I$ satisfies [A]. Then the following comparison properties hold:
\begin{enumerate}
    \item Let $x,y\in [0,\infty)$, \begin{equation}
\textup{(CP1)} \quad x \leq y \ \Rightarrow \ \mathbb{P}\big( X_t^x\leq X_t^y, \ \forall t\in [0,\zeta(x)\wedge\zeta(y))\big)=1
\end{equation}
\item Let $x\in [0,\infty)$ and  $I^{(1)}$ and $I^{(2)}$ be two functions defined on $[0,\infty)$, satisfying [A] and $X^{(1)}$, $X^{(2)}$ be the CBDI($\Psi$, $I^{(1)}$) and CBDI($\Psi$, $I^{(2)}$) started from a same value $x\in (0,\infty)$. We have:
\begin{equation}
\textup{(CP2)} \quad I^{(1)} \leq I^{(2)} \ \Rightarrow \ \mathbb{P}\big(X^{(2)}_t \leq X^{(1)}_t,\ \forall t\in [0,\zeta^{(1)}\wedge\zeta^{(2)})\big)=1.
\end{equation}
\end{enumerate}
\end{proposition}
\smallskip

By the comparison property (CP1), one can define  the process $X$ starting from infinity, denoted by  $X^{\infty}$, as the following almost sure pointwise increasing limit:
\begin{equation}
\forall t \geq 0, \quad X^{\infty}_t := \lim_{x \to \infty} \uparrow \, X^x_t \quad \textup{a.s.}
\end{equation}
Note that $X^{\infty}_0=\infty$ a.s. and that $X^{\infty}_t$ may well be infinite for $t>0$.  
\vspace{2mm}
\\We denote by $\mathbb{P}_{\infty}$, the law of the process $\left(X^\infty_t\right)_{t \geq 0}$ and we say that the process $X$ \emph{comes down from infinity}, or equivalently that $\infty$ is an instantaneous entrance point for $X$, if
\[
\forall t > 0, \quad X^{\infty}_t < \infty \quad \textup{a.s.}
\]

We assume from now on that, when the process has jumps, i.e. $\pi\not{\equiv}0$, one has $\overline{\pi}(1)>0$, i.e. we are not restricted to a case with only small jumps. As we are interested here in processes with no explosion, we also focus on the case with no jump to $\infty$, i.e. $\lambda=\pi(\{\infty\})=0$. In particular, this entails that $\underset{x\rightarrow \infty}{\lim} \downarrow \bar{\pi}(x)=0$.

% \subsection{The setting of classical CBs ($I\equiv 0$)}
% (on parle quand même d'extinction ici ?) \\
% The phenomena of explosion and extinction, namely the fact that the process can reach $\infty$ or $0$ in finite time with positive probability are completely understood in the branching setting. We recall the classification and refer the reader for instance to Kyprianou's book \cite[Theorems 12.3 and 12.5]{Kyprianou}. Recall $\Psi$ the branching mechanism \eqref{eq LK}.
% \\

% The CB$(\Psi)$ process does not explode almost surely if and only if 
% \begin{equation}
% \label{eq dynkin}
%     \int_0 \frac{du}{|\Psi(u)|} = \infty \qquad \text{(Dynkin's condition).}
% \end{equation}
% A sufficient condition (not necessary) for having no explosion is $\Psi'(0+)\in (-\infty,\infty)$. Equivalently, the Lévy measure $\pi$ has a first moment: $\int^{\infty}_1 h\pi(\ddr h)<\infty$. An example of explosive CB process is the subordinator stable case  $\Psi(u)=-u^{\alpha}$ for $\alpha\in (0,1)$.\\

% The CB$(\Psi)$ process gets extinct almost surely if and only if 
% \begin{equation}
%     \label{eq grey}
%     \int^\infty \frac{du}{|\Psi(u)|} < \infty \qquad \text{(Grey's condition).}
% \end{equation}
% In case \eqref{eq grey} holds, necessarily $\Psi$ takes positive values in a neighbourhood of $\infty$. A sufficient condition (not necessary) for extinction is $\sigma=\underset{x\rightarrow \infty}{\lim} \frac{\Psi(x)}{x^2}>0$. The CB process with branching mechanism $\Psi:x\mapsto x\log x$ (Neveu CB process) does not get extinct in finite time and does not explode.
\section{Main results}
\label{section:main}
%Recall the stochastic differential equation \eqref{eds}. The drift interaction term is governed by $-I(X_t)\ddr t$. We first give condition on $I$ such that the stochastic equation \eqref{eds} admits a unique strong solution. This is the 

% For the moment, we admit that if $I$ satisfies [A], \eqref{eds} has a unique strong solution $X$ and that $\left(X(x)\right)_{x\in(0,\infty)}$ is a.s. nondecreasing. (pas besoin de ça car on est dans les main results on fait qu'énoncer ?)

For the remainder of this work, we take $I:\mathbb{R}_{+}\mapsto \mathbb{R}$ satisfying [A] and consider $X$ the $\mathrm{CBDI}(\Psi,I)$, i.e. the unique minimal solution to \eqref{eds}. 
%\subsection{Non-explosion and coming down from infinity}
Our first result establishes sufficient conditions for non-explosion and for coming down from infinity. To this end, we introduce an additional set of technical conditions on $I$, denoted by [B]. These conditions are specifically tailored for the construction of Lyapunov functions and for the application of Theorem \ref{th cond lyapunov}.
%Our first result will provide sufficient conditions for non-explosion and coming down from infinity. We shall need 
%To analyze the behavior at the boundary $\infty$ of the CBDI, we introduce 
%an additional set of technical conditions over $I$, denoted by [B]. The latter are taylor-made for designing Lyapunov functions and applying Theorem \ref{th cond lyapunov}. 
% In this section we focus on defining and describing the process starting from $\infty$, denoted $X^{\infty}$. The comparison property (CP1) in Proposition \ref{prop comp} enables its construction as the almost sure pointwise limit of $X(x)$ as $x\rightarrow \infty$. We establish two main results. 

% First, Theorem \ref{th descente inf} provides two sufficient conditions : one ensuring the non-explosion of the CBDI, and another guaranteeing that the process comes down from infinity. The second result, Theorem \ref{th descente eds} concerns the process $X^{\infty}$. Under an additional one-sided Lipschitz condition on $I$, this theorem yields a stronger mode of convergence of $X(x)$ towards $X^{\infty}$, and establishes the fundamental property that the process starting from infinity is the unique strong solution to a stochastic differential equation with initial state $\infty$.

\paragraph{Condition [B]} \textit{There exists a constant $\kappa \geq 0$ such that:
\begin{itemize}
\item[(B1)] $I$ is $C^1$ and positive on $[\kappa,\infty)$, $z\mapsto \frac{I(z)}{z}$ is nondecreasing on $[\kappa,\infty)$, $\underset{z\rightarrow \infty}{\lim} \frac{I(z)}{z}=\infty$ and $\int_{\kappa}^{\infty}\frac{u}{I(u)}\ddr u=\infty$.
\item[(B2)] 
%$I$ is $C^1$ and positive on $[\kappa,\infty)$ and 
The map $z\mapsto \frac{I'(z)}{z}$ is bounded on $[\kappa,\infty)$.
\item[(B3)] There exists a constant $b>0$ such that \[\forall (y,z)\in [0,\infty)^2, \ I(y+z)-I(y)\geq -bz.\] We call this condition the one-sided Lipschitz condition.
\end{itemize}}
%\paragraph{}
% By the comparison property (CP1), one can define  the process $X$ starting from infinity, denoted by  $X^{\infty}$, as the following almost sure pointwise increasing limit :
% \begin{equation}
% \forall t \geq 0, \quad X^{\infty}_t := \lim_{x \to \infty} \uparrow \, X_t(x) \quad \textup{a.s.}
% \end{equation}
% Note that $X_0(\infty)=\infty$ a.s. and that $X^{\infty}_t$ may well be infinite for $t>0$.  
% \vspace{2mm}

% We denote by $\mathbb{P}_{\infty}$, the law of the process $\left(X^{\infty}_t\right)_{t \geq 0}$ and we say that the process $X$ \emph{comes down from infinity}, or equivalently that $\infty$ is an instantaneous entrance point for $X$, if
% \[
% \forall t > 0, \quad X^{\infty}_t < \infty \quad \textup{a.s.}
% \]
%{\color{red} est ce que mettre (11) dans B1 ne serait pas mieux?} on a seulement besoin de cette hypothèse pour le point (i).
\begin{theorem}\label{th descente inf} Let $X$ be a $\mathrm{CBDI}(\Psi,I)$.  %and that 
%\begin{equation}
%\label{hyp desc}
%\int_{\kappa}^{\infty}\frac{u}{I(u)}\ddr u=\infty.
%\end{equation}

\begin{itemize}
\item[(i)] If $I$ satisfies (B1) and $\mathcal{I}:=\int_{\kappa}^{\infty}\frac{u\overline{\pi}(u)}{I(u)}\ddr u<\infty$, then $X$ almost surely does not explode, namely $$\forall x\in (0,\infty), \quad \mathbb{P}_x(\tau_{\infty}^{+}<\infty)=0,$$ and its first passage times below a certain level are integrable, i.e. $$\exists x_0\in (0,\infty) ; \ \forall x>x_0, \ \mathbb{E}_x[\tau_{x_0}^{-}]<\infty.$$
\item[(ii)] If $I$ satisfies (B1), (B2) and $\mathcal{J}:=\int_{\kappa}^{\infty}\frac{1+u\overline{\pi}(u)}{I(u)}\ddr u<\infty$, then $X$ does not explode and comes down from infinity. Furthermore,
$$\exists x_0\in (0,\infty); \ \mathbb{E}_{\infty}[\tau_{x_0}^{-}]<\infty.$$
\end{itemize}
\end{theorem}
The proof of Theorem \ref{th descente inf} can be found in Section \ref{section:lyapfun}.\\

%\begin{remark}
The integral condition $\mathcal{J}<\infty$, stated in Theorem~\ref{th descente inf}~(ii), is equivalent to the finiteness of the following two integrals:
\[
\int_{\kappa}^{\infty}\frac{\ddr u}{I(u)} < \infty \quad \text{and} \quad \mathcal{I}=\int_{\kappa}^{\infty}\frac{u\,\bar{\pi}(u)}{I(u)} \, \ddr u < \infty.
\]
The first integrability condition ensures that the deterministic flow, started from $x_0=\infty$,
\[\frac{\ddr x_t}{\ddr t}=-I(x_t), \ x_0=\infty, \ t\in [0,\infty),\] comes down from infinity, namely $x_t<\infty$ for all $t>0$. Indeed, by rewriting the differential equation in its integral form, $(x_t)$ such that $x_0=\infty$ satisfies for all $t>0$, $$\int_{x_t}^{\infty}\frac{\ddr u}{I(u)}=t\in (0,\infty).$$ When the integral on the left hand side is finite, we see that  for all $t>0$, necessarily $x_t<\infty$.  
\vspace{1mm}

The second integrability condition $\mathcal{I}<\infty$
%\[
% \int_{\kappa}^{\infty}\frac{u\,\bar{\pi}(u)}{I(u)}\,\ddr u,
% \]
captures a form of stability for the CBDI process when started from a finite value. %It reflects the interplay between the jump measure $\pi$ and the drift term $I$. 
%The finiteness of this integral indicates that the drift must counterbalance the reproduction term $u\bar{\pi}(u)$ strongly enough near infinity to prevent the process from staying infinite. This highlights the crucial role of the function $I$ in the coming down from infinity phenomenon.
%\vspace{2mm}
% The next proposition provides an equivalent formulation of the integral condition in Theorem~\ref{th descente inf}~(ii), under the assumption that the function $I$ has regular variation.
% \begin{proposition}
% Assume that $\Psi$ has regular variation, the integral condition appearing in Theorem \ref{th descente inf} (ii) becomes equivalent to :
% $$\int_{\kappa}^{\infty}\frac{1+u\mid \Psi(1/u)\mid}{I(u)}du<\infty.$$
% \end{proposition}
% \paragraph{}
% We now present examples of mechanisms to which our results apply. As a preliminary step toward a forthcoming work \cite{foucartrebotier}, we treat below a specific class of drift-interaction mechanisms.
\newpage
By applying Theorem \ref{th descente inf}, we have the following.
\begin{example}
Let $\pi$ and $I$ be satisfying: 
\begin{align*}
&\bar{\pi}(z) \underset{z\rightarrow \infty}{\sim} c_B z^{-\alpha}(\log z)^{\beta} \text{ and } \ I(z) \underset{z\rightarrow \infty}{\sim} c_I z^{\hat{\alpha}}\left(\log z\right)^{\hat{\beta}},\end{align*} 
with $\alpha \in (0,2], \  \beta, \hat{\alpha}, \hat{\beta} \in (-\infty,\infty)$, $c_B,c_I>0$.  

\begin{table}[h!]
\centering
\renewcommand{\arraystretch}{1.3}
\begin{tabular}{|p{0.47\textwidth}|p{0.47\textwidth}|}
\hline
\centering \textbf{Non-explosion} \\$(B1),\ \mathcal{I}<\infty$ & 
\centering \textbf{Coming down from infinity} $(B1),(B2),\ \mathcal{J}<\infty$ \tabularnewline
\hline
(a1) $\alpha+\hat{\alpha}>2$, $\hat{\alpha}\in(1,2)$
& (b1) $\alpha+\hat{\alpha}>2$, $\hat{\alpha}\in(1,2)$ \\[6pt]
\hline
(a2) $\alpha>1$, $\hat{\alpha}=1$, $\hat{\beta}>0$
& (b2) $\alpha>1$, $\hat{\alpha}=1$, $\hat{\beta}>1$ \\[6pt]
\hline
(a3) $\hat{\alpha}=2$
& (b3) $\hat{\alpha}=2$ \\[6pt]
\hline
(a4) $\alpha+\hat{\alpha}=2$, $\hat{\alpha}\in(1,2)$, $\hat{\beta}-\beta>1$
& (b4) $\alpha+\hat{\alpha}=2$, $\hat{\alpha}\in(1,2)$, $\hat{\beta}-\beta>1$ \\[6pt]
\hline
(a5) $\alpha=1$, $\hat{\alpha}=1$, $\hat{\beta}-\beta>1$, $\hat{\beta}>0$
& (b5) $\alpha=1$, $\hat{\alpha}=1$, $\hat{\beta}-\beta>1$, $\hat{\beta}>1$ \\[6pt]
\hline
\end{tabular}
\end{table}

\end{example}

Note that conditions (a3) and (b3) follow directly from the combination of Theorem \ref{th descente inf} and the comparison property (CP2) stated in Proposition \ref{prop comp}; all other conditions follow from the theorem alone.
This completes the classification established in \cite[page 2536]{LYZ} for polynomial branching processes. In particular, the setting $\alpha \in (0,1)$, for which the classical CB process explodes, see e.g. \cite[Chapter 12]{Kyprianou}, is also covered here in cases (a1)-(a3)-(a4), (b1)-(b3)-(b4).

The next proposition provides equivalent formulations for the integral condition in Theorem~\ref{th descente inf}~(ii) to hold. In particular, the condition $\mathcal{I}<\infty$ can be rewritten as a certain moment condition on the Lévy measure $\pi$.

\begin{proposition}\label{prop:equivond}
Recall $\mathcal{I}$ and $\mathcal{J}$.
%Set $\mathcal{I}:=\int_{\kappa}^{\infty}\frac{u\,\bar{\pi}(u)}{I(u)} \, du$.
\begin{enumerate}
    \item  One has 
    %The condition $\mathcal{I}<\infty$ is equivalent to a moment condition on the Lévy measure $\pi$ :
\[ \mathcal{I}<\infty \text{ iff }\int_{\kappa}^{\infty} G(h) \, \pi(\ddr h) < \infty,
\text{ with } G(h) := \int_{\kappa}^h \frac{u}{I(u)} \, \ddr u\geq 0,\ h\in [\kappa,\infty).\]
% \item Under the condition $\underset{u\rightarrow \infty}{\lim}\!\! \uparrow I(u)/u=\infty$. One has \[ \mathcal{I}<\infty \text{ iff }\int_{\kappa}^{\infty} \left\lvert \left(\frac{u}{I(u)}\right)'\Psi(1/u)u \right \lvert du<\infty \text{ and } \underset{u\rightarrow \infty}{\limsup}\big \lvert \frac{u^2}{I(u)}\Psi(1/u)\big \lvert<\infty\]

\item Assume that $\Psi$ is regularly varying at $0$ with index $\alpha\in [0,1)$, namely $\Psi(cx)/\Psi(x) \underset{x\rightarrow 0}{\rightarrow} c^{\alpha}$, then 
%then integral condition appearing in Theorem \ref{th descente inf} (ii) becomes equivalent to :
$$\mathcal{J}<\infty \text{ iff } \int_{\kappa}^{\infty}\frac{1+u|\Psi(1/u)|}{I(u)}\ddr u<\infty.$$
\end{enumerate}
  Assumption (B1) ensures that there exists $C\in (0,\infty)$, such that $G(h)\leq Ch$ for all $h\in [ \kappa,\infty)$. As a direct consequence, we find that any $\mathrm{CBDI}(\Psi,I)$ such that
\[-\Psi'(0+)=\int^\infty_{1}h\pi(\ddr h)<\infty \text{ and } \int_{\kappa}^\infty \frac{\ddr u}{I(u)}<\infty\]
 comes down from infinity. This recovers a result obtained in \cite{LePardoux}. \\ In the logistic case, that is when $I(x) := \frac{c}{2}x^2$ for $x\in [0,\infty)$, with $c>0$, we have $G(h)\underset{z\rightarrow \infty}{\sim} \log h$ and the moment condition becomes 
\[
\int_{\kappa}^{\infty} \log(h) \, \pi(\ddr h) < \infty,
\]
which is known to ensure coming down from infinity of the logistic CB process, see \cite[Corollary 3.10]{Lambert}.

\end{proposition}
\begin{proof}
For 1. By Fubini-Tonelli's theorem, one has :
\[\int_{\kappa}^{\infty}\frac{u\,\bar{\pi}(u)}{I(u)}\ddr u =\int_{\kappa}^{\infty} \int_{\kappa}^h \frac{u}{I(u)} \, \ddr u \, \pi(\ddr h) < \infty.
%\text{ with } G(h) = \int_{\kappa}^h \frac{u}{I(u)} \, \ddr u,\ h\in [\kappa,\infty).
\] 
For 2. For $\alpha\in [0,1)$. The regular variation property yields the equivalence $$\bar\pi(u)\underset{u\rightarrow \infty}{\sim} -\frac{1}{\Gamma(1-\alpha)}\Psi(1/u).$$
We refer e.g. to Li et al. \cite[Lemma 2.3]{LFZ24} and the references therein.
%We apply Proposition 1 page 74 of Bertoin's book \cite{}. Attention c'est écrit pour les subordinateurs. Il faut donc adapter, on a toujours la décomposition $Psi(x)=\Sigma(x)-\Phi(x)$ avec $\Sigma$ un mécanisme de branchement (sous)critique et $\Phi$ une fonction de Bernstein. Tu dois pouvoir controler Sigma. 
\end{proof}

\paragraph{}
The next result establishes a form of regularity of the CBDI process with respect to its initial value, namely local uniform convergence. For this purpose, we assume the additional condition (B3) on the drift function \(I\). A related result can be found in \cite[Proposition~1.6, p.~5]{LWZqsd}, where the drift function is assumed to be non-decreasing. By contrast, our approach only relies on the non-explosiveness of the CBDI process, and therefore does not require any monotonicity assumption.

\begin{proposition}\label{prop cv unif}
    Assume that $I$ satisfies (B3) and that $X$ is non-explosive. Then for all $y \in[0,\infty)$, all $t\geq 0$ and all $(y_n)_{n\in \mathbb{N}}$ sequence such that $y_n\rightarrow y$ when $n \rightarrow\infty$, one has:
    \begin{equation*}
        \left(X_s^{y_n}, \ 0\leq s\leq t\right)\underset{n\rightarrow\infty}{\longrightarrow} \left(X_s^{y}, \ 0\leq s\leq t\right) \quad \mathbb{P}\textup{-a.s. in } (D,\rho_{\infty}).
    \end{equation*}
    Furthermore, when $I(0)=0$, for all sequence $(x_n)$ decreasing to $0$, one has $$X_t^{0+}:=\underset{n\rightarrow\infty}{\lim} X_t^{x_n}=X_t^0=0,\ \forall t\geq 0 \text{ a.s.}.$$
\end{proposition}

We now focus on the case where the CBDI process comes down from infinity and the drift function $I$ satisfies $(B3)$. 
%Assuming an additional regularity condition on , . 
Under these assumptions, the following theorem shows that the process started from infinity satisfies a stochastic differential equation. This representation guarantees also that $X^{\infty}$ is a càdlàg strong Markov process.

\begin{theorem}
\label{th descente eds}
Assume that $I$ satisfies the one-sided Lipschitz condition (B3), that the CBDI is non-explosive and that for some $a>0$, $\tau_a^{-}(\infty)<\infty$ a.s.. Then for all $t\in [0,\infty)$: 
\begin{equation} \label{eq cvu th desc}
    \left(X^x_s, \ 0\leq s\leq t\right) \ \underset{x\rightarrow \infty}{\longrightarrow} \ \left(X^{\infty}_s, \ 0\leq s\leq t\right) \quad \textup{a.s. in} \ (D,\rho_{\infty}).
\end{equation} Moreover $(X^{\infty}_t)_{t\geq 0}$ is the unique non-explosive strong solution of the stochastic equation:
\begin{equation}
\begin{split}
X_t=X_r+ \sigma\int_r^t\sqrt{X_s}\ddr B_s+ \int_r^t\int_0^{X_{s-}}\int_0^1h\tilde{M}(\ddr s,\ddr u,\ddr h)&+\int_r^t\int_0^{X_{s-}}  \int_1^{\infty}  h M(\ddr s,\ddr u,\ddr h)
\\&\quad -\int_r^tI(X_s)\ddr s, \ \quad 0< r\leq t,
\end{split}
\end{equation}
such that $X^{\infty}_t \underset{t\downarrow 0}{\rightarrow} \infty$ a.s..\\
Additionally, if $(X_t^\infty)_{t\geq 0}$ gets absorbed at $0$ in finite time a.s. then the convergence in $(D,\rho_{\infty})$ is true globally, namely \eqref{eq cvu th desc} holds with $t=\infty$.
\end{theorem}
 In particular, Theorem~\ref{th descente eds} applies whenever the conditions (B1), (B2), and (B3) are satisfied and $\mathcal{J}<\infty$, in which case $\infty$ is a continuous entrance boundary. Moreover, the local uniform convergence of $(X^x_t,\ t\geq 0)$ established in Theorem~\ref{th descente eds}, together with the absence of negative jumps in $X^x$, implies in particular that the process $X^{\infty}$ also has no negative jumps.\\ 

 The proofs of Proposition \ref{prop cv unif} and Theorem~\ref{th descente eds} are given in Section  \ref{section:edsinf}.

\section{First passage times and Lyapunov functions}\label{sec:lyapunov}
In this section, we place ourselves in a broader setting than CBDIs and consider a  $[0,\infty]$-valued process $X$ with both boundaries absorbing and subject to several assumptions (all will be met for CBDIs). 
\vspace{2mm}

We start by collecting basic properties of first passage times. 
%\subsection{First passage times and their limits}
Assume given processes $(X^x_t,t\geq 0)$ for $x\geq 0$, such that 
\begin{enumerate}
    \item[a)] 
For all $x\geq 0$, a.s. $X^x_0=x$, $(X^x_t,t\geq 0)$ has càdlàg paths.
%with no negative jumps
    \item[b)] $X$ satisfies (CP1).
\end{enumerate}
Recall that the comparison property (CP1) allows us to define the process $(X_t^\infty, t\geq 0)$ as the almost sure nondecreasing pointwise limit of $(X_t^x, t\geq 0)$ as $x$ goes to $\infty$ and that we call $\mathbb{P}_\infty$ the law of $(X_t^\infty, t\geq 0)$.

The following lemma provides basic properties on the limit of the first passage times $\tau_{a/b}^{-/+}(x)$. Naturally, the limit when $x$ tends to $\infty$ of the first passage time below $a$ of the CBDI starting from $x$ is the first passage time below $a$ of the CBDI starting from infinity. Note that when the CBDI process does not come from infinity, this limit is infinite.

\begin{lemma}
\label{lemma stopping times} 
The first passage times satisfy the following properties :
\begin{itemize}
\item[(i)] $\tau_b^{x,+} \uparrow \tau_{\infty}^{x,+}:=\inf\{t\geq 0, \ X^x_t=\infty  \ \textit{ or } \ X^x_{t-}=\infty\}$ a.s. when $b \rightarrow \infty$.
\item[(ii)] $\tau_a^{x,-} \uparrow \tau_{a}^{\infty,-}:=\inf\{t\geq 0,\ X^{\infty}_t\leq a\}$ a.s. when $x \rightarrow \infty$.
\item[(iii)]
$\tau_{a}^{\infty,-} \downarrow \tau_{\infty}^{\infty,-}:=\inf\{t\geq 0, \ X^{\infty}_t<\infty\}$ a.s. when $a \rightarrow \infty$.
\end{itemize}
\end{lemma}
\begin{proof}
\textbf{(i)} Fix $x > 0$. Since the family of random times $\left(\tau_b^{x,+}\right)_{b > x}$ is nondecreasing almost surely in $b$, we define
\[
\ell := \lim_{b \to \infty} \tau_b^{x,+}, \quad \text{a.s.}.
\]
By definition of $\tau_\infty^{x,+}$ and the fact that $\{X_t^x = \infty\} \subset \{X_t^x \geq b\}$ for all $b > 0$, we immediately obtain:
\[
\tau_\infty^{x,+} \geq \ell \quad \text{a.s.}.
\]
Since the process $X$ is càdlàg, we have:
\[
X_{\tau_b^{x,+}}^x \longrightarrow X_{\ell-}^x \quad \text{a.s.}, \quad \text{as } b \to \infty.
\]
Moreover, by definition of $\tau_b^{x,+}$, we have $X_{\tau_b^{x,+}}^x \geq b$ a.s. for all $b \geq x$. Taking the limit $b \to \infty$ yields
\[
X_{\ell-}^x = \infty \quad \text{a.s.}.
\]
Since $\{X_{t-}^x = \infty\} \subset \{X_t^x = \infty\}$, it follows that $\ell \geq \tau_\infty^{x,+}$ a.s.. Hence, by combining both inequalities
\[
\tau_\infty^{x,+} = \ell \quad \text{a.s.}
\]

\medskip
\textbf{(ii)} Fix $a > 0$. Since $X_t^x \uparrow X^{\infty}_t$ for all $t\geq 0$ a.s. as $x \to \infty$, it follows that, as $x$ goes to $\infty$:
\[
%\mathbf{1}_{\{X_s^x > a\}} \longrightarrow \mathbf{1}_{\{X^{\infty}_s > a\}} \quad \text{a.s.}, \quad \forall s \leq t.
\mathbb{P}\left(X_s^x>a\right) \longrightarrow \mathbb{P}\left(X^{\infty}_s>a\right) 
%\quad \text{a.s.}
, \quad \forall s \leq t.
\]
Properties a) and b) imply that:
\[
%\mathbf{1}_{\{\forall s \leq t,\, X_s^x > a\}} \longrightarrow \mathbf{1}_{\{\forall s \leq t,\, X^{\infty}_s > a\}} \quad \text{a.s.}, \quad \text{as } x \to \infty.
\mathbb{P}\left(\forall s\leq t, \ X_s^x>a\right) \longrightarrow \mathbb{P}\left(\forall s\leq t , \ X^{\infty}_s>a\right) %\quad \text{a.s.}, 
\quad \text{as } x \to \infty.
\]

Thus, as $x$ goes to $\infty$:
\[
\mathbb{P}(\tau_a^{x,-} > t) = \mathbb{P}(\forall s \leq t,\, X_s^x > a) \longrightarrow \mathbb{P}(\forall s \leq t,\, X^{\infty}_s > a) = \mathbb{P}(\tau_a^{\infty,-} > t).
\]
Since $x \mapsto \tau_a^{x,-}$ is a.s. non-decreasing, it follows that:
\[
\tau_a^{x,-} \uparrow \tau_a^{\infty,-}\quad \text{a.s.}, \quad \text{as } x \to \infty.
\]

\medskip
\textbf{(iii)} The family $\left(\tau_a^{\infty,-}\right)_{a > 0}$ is a.s. non-increasing in $a$, so we define:
\[
L := \lim_{a \to \infty} \tau_a^{\infty,-} \quad \text{a.s.}.
\]
Fix $t \geq 0$. Note that $\{X^{\infty}_t \leq a\} \subset \{X^{\infty}_t \leq \infty\}$ for all $a > 0$, so :
%\[
$\tau_\infty^{\infty,-} \leq L \quad \text{a.s.}.$
%\]
Furthermore, we have:
\[
\left\{ \forall a > 0,\, t < \tau_a^{\infty,-} \right\}
= \left\{ \forall a > 0,\, X^{\infty}_t > a \right\}
\subset \left\{ X^{\infty}_t = \infty \right\}.
\]
Therefore, on the event $\{X^{\infty}_t < \infty\}$, there exists $a > 0$ such that $\tau_a^{\infty,-} \leq t$ a.s. Since $L \leq \tau_a^{\infty,-}$ a.s. for all $a > 0$, it follows that on $\{X^{\infty}_t < \infty\}$:
$L \leq t$  a.s.. Thus, $L \leq \tau_\infty^{\infty,-}$ \text{a.s.} and by combining the bounds, we get:
\[
\tau_\infty^{\infty,-} = \lim_{a \to \infty} \tau_a^{\infty,-} = L \quad \text{a.s.}.
\]
\end{proof}

We now give some background on the method of Lyapunov functions. Generally speaking, a Lyapunov function for a process $X$ or its (extended) generator $\mathcal{X}$ is a function allowing one to study the behavior of the process near a boundary point. We focus here on the study of the boundary $\infty$ and give sufficient conditions guaranteeing that the process does not explode and comes down from infinity. These conditions are well known in the literature, especially in the setting of discrete-state Markov processes. We refer for instance to Chow and Khasminskii \cite{Chow} and Menshikov and Petretis \cite{Petritis}. By taking advantage of the comparison property (CP1), we design conditions involving only the behavior of the extended generator in a neighborhood of $\infty$.
\\

We assume additionally, to the previous hypotheses a) and b):
\begin{enumerate}
    \item[c)] 
The process $X$ is strong Markov.
\item[d)] It solves a local martingale problem of the form:
\begin{equation}
\label{eq pm}
\forall f\in D_\mathcal{X}, \quad \left(f(X_t)-\int_0^t\mathcal{X}f(X_s)\ddr s\right)_{t\geq 0} \quad \textup{is a local martingale}
\end{equation}
for some operator $\mathcal{X}$ with domain $D_\mathcal{X}:=\left\{f \text{ real function }:\mathcal{X}f(x) \text{ well-defined for all }x\in D_f\right\}$.
%\item For simplification, we also assume that $X$ satisfies the comparison property (CP1), meaning that for all $x,y\in (0,\infty)$, $$ x\leq y \Rightarrow \left(\forall t\geq 0, \ X_t(x)\leq X_t(y)\right) \quad \textup{a.s.}.$$
\end{enumerate}
%We refer the reader for instance to \cite[Proposition 1]{PalauPardo} for the form that the operator $\mathcal{X}$ can take when the process $X$ is solution to a stochastic equation with jumps.
%Recall that under these conditions, we can define a process starting from infinity, denoted $(X^{\infty}_t)_{t\geq 0}$, as the a.s. pointwise limit of $(X_t(x))_{t\geq 0}$ when $x \rightarrow \infty$ :
%$$\forall t\geq 0, \quad X^{\infty}_t:= \ \uparrow \underset{x\rightarrow \infty}{\lim} X_t(x). $$
%\paragraph*{} 
%We say that $X$ comes down from infinity, or equivalently that $\infty$ is an instantaneous entrance point for $X$, if $$\forall t> 0, \ X^{\infty}_t<\infty \quad \textup{a.s.}. $$
%We continue to use the notations $\tau_{a/b}^{-/+}$ for the hitting times introduced previously.
%The proof of condition (iii) was provided in \cite{FZfragcoag} for the case of a discrete-valued Markov process. For further details on the discrete-state case, we refer the reader to the work of Petritis et al. in .

\begin{thA}\label{th cond lyapunov} Assume $\textrm{a)-b)-c)-d).}$
\begin{itemize}
\item[(i)] Assume that there exist $x_0\geq 0$, $c>0$ and $f\in D_\mathcal{X}$ nonnegative, non-decreasing such that $f(x) \rightarrow \infty$ when $x\rightarrow \infty$ and
\begin{equation}
\mathcal{X}f(x)\leq cf(x), \ \ \forall x\geq x_0.
\end{equation}
Then, $\mathbb{P}_x(\tau_{\infty}^{+}=\infty)=1, \ \forall x\in (0,\infty).$
\item[(ii)] Assume that there exist $c>0, x_0>0$ and $f\in D_\mathcal{X}$ nonnegative, non-decreasing such that $f(x)\rightarrow \infty$ when $x\rightarrow \infty$ and
\begin{equation}
\mathcal{X}f(x)\leq -c, \ \ \forall x\geq x_0.
\end{equation}
Then, 
%the first passage times of $X$ under $x_0$ are integrable, i.e. 
$\mathbb{E}_x\big[\tau_{x_0}^{-}\big]<\infty, \ \forall x\geq x_0.$
\item[(iii)] Assume that $X$ is non-explosive and that there exist $c>0, x_0>0$ and $f\in D_\mathcal{X}$ nonnegative, non-decreasing and bounded such that\begin{equation}
\mathcal{X}f(x)\leq -c, \ \ \forall x\geq x_0.
\end{equation}
Then, $\mathbb{E}_{\infty}\big[\tau_{x_0}^{-}\big]<\infty,$
% $\sup_{x\geq x_0}\mathbb{E}_x\big[\tau_{x_0}^{-}\big]<\infty,$ and $\infty$ is an instantaneous entrance boundary in the following sense:  
% \begin{equation}\label{eq:entranceboundary}\forall t>0, \ \underset{x_0\rightarrow \infty}{\lim}\underset{x\rightarrow \infty}{\liminf}\ \mathbb{P}_x(\tau_{x_0}^{-}\leq t)=1.
% \end{equation}
and $X$ comes down from infinity, i.e. 
\begin{equation}\label{eq:cdi}\mathbb{P}_{\infty}\left(\forall t> 0, \ X_t<\infty\right)=1. \end{equation}
\end{itemize}
\end{thA}

%The following lemma presents fundamental results on first passage times and their limits. A detailed proof of this lemma can be found in the appendix.
\begin{proof}
The general structure of the proof involves considering either the space-time process $(X_t,t)_{t\geq 0}$ or the process $X$ itself, and stopping the associated local martingale at the stopping time $\tau_a^{-}\wedge \tau_b^{+}$. 
%For the last two conditions, we bound the first moment of $\tau_a^{-}$ using  Lyapunov function. 

%We recall that Theorem A provides criteria involving Lyapunov functions $f$ and the extended generator $\mathcal{X}$, allowing us to deduce information about the explosion and first passage times of the process $X$. These criteria yield (super)martingale properties, which we analyze via the optional stopping theorem.
%(dire qq part que $f(X_t)$ est une semimartingale ecire la decomp peut etre)
\paragraph{(i)} Let $f \in \mathcal{D}_\mathcal{X}$ be a nonnegative, non-decreasing function such that $f(z) \to \infty$ as $z \to \infty$ and $\mathcal{X}f(x) \leq cf(z)$ for all $z \geq x_0$ and some constant $c>0$ and some $x_0\geq 0$. Define the function $F(z,t) := f(z)e^{-ct}$ and let $(Y_t,t\geq 0):=(X_{t\wedge \tau_{x_0}^{-}},t\geq 0)$ be the process stopped at $\tau_{x_0}^{-}$. Assume that the initial value of the process $x$ is in $(x_0,\infty)$.

% By applying It\^o's formula (see e.g. \cite[page 63]{Situ}) to $F(Y_t, t)$, we obtain :
% \begin{align*}
% F(Y_t,t) - f(x) &= \int_0^t f'(Y_{s-})e^{-cs} \sqrt{Y_{s-}} \, \ddr B_s - \int_0^t f'(Y_{s-})e^{-cs} I(Y_{s-}) \, \ddr s \\
% &\quad - c \int_0^t f(Y_{s-})e^{-cs} \, ds + \frac{\sigma^2}{2} \int_0^t f''(Y_{s-})e^{-cs} Y_{s-} \, ds \\
% &\quad + \int_0^t \int_0^{Y_{s-}} \int_0^{\infty} \left[f(Y_{s-}+h) - f(Y_{s-}) - f'(Y_{s-})h\mathbf{1}_{\{h\leq 1\}}\right] e^{-cs} \, M(ds,dz,\ddr h) \\
% &= \int_0^t (\mathcal{X}f(Y_s) - cf(Y_s))e^{-cs} \, ds + \mathcal{M}_t,
% \end{align*}
% where $\mathcal{M}_t$ is a local martingale and $\mathcal{A}F(z,t) := (\mathcal{X}f(z) - cf(z))e^{-ct} \leq 0$ for $z\geq x_0$.
The process $(f(Y_t),t\geq 0)$ being a semimartingale, we obtain by integration by parts
\begin{align*}
e^{-ct}f(Y_t) - f(x) &= \int_0^t e^{-cs}\ddr f(Y_s)-c\int_0^{t}e^{-cs}f(Y_s) \ddr s\\
&= \int_{0}^{t}e^{-cs}\ddr\Big(f(Y_s)-\int_0^s\mathcal{X}f(Y_u)\ddr u \Big)+\int_0^t (\mathcal{X}f(Y_s) - cf(Y_s))e^{-cs} \, \ddr s ,
\end{align*}
where the process $(\mathcal{M}_t,t\geq 0)$ defined  for all $t\geq 0$ by $$\mathcal{M}_t:=\int_{0}^{t}e^{-cs}\ddr \Big(f(Y_s)-\int_0^t\mathcal{X}f(Y_u)\ddr u \Big),$$ is a local martingale. By assumption on $f$, $f\geq 0$ and $\mathcal{X}f(z) - cf(z)\leq 0$ for $z\geq x_0$, hence, 
\[\Big(e^{-ct}f(Y_t)-\int_0^t (\mathcal{X}f(Y_s) - cf(Y_s))e^{-cs} \, \ddr s  , \ t> 0\Big)\]
%$\Big(F(Y_t,t) - \int_0^t \mathcal{A}F(Y_s,s) \, \ddr s, \ t\geq 0\Big)$ 
is a nonnegative local martingale, and thus a supermartingale. 
For any bounded stopping time $T$,
\begin{equation}\label{ineq-thA-i}
\mathbb{E}_x\left[f(Y_T)e^{-cT}\right] \leq f(x).
\end{equation}

Assume by contradiction that $Y$ explodes, i.e. $\mathbb{P}_x(\sigma_{\infty}^+ < \infty) > 0$, where $\sigma_{\infty}^+$ is the explosion time of $Y$. Taking $T = t\wedge \sigma_{\infty}^+$ with $t\geq 0$ in \eqref{ineq-thA-i}, we get:
$$\mathbb{E}_x\left[f\left(Y_{t\wedge \sigma_{\infty}^{+}}\right)e^{-c\left(t\wedge \sigma_{\infty}^{+}\right)}\right]\leq f(x).$$
Take $t\geq 0$ such that $\mathbb{P}_x(\sigma_{\infty}^{+}<t)>0$. We have :
$$\infty=\mathbb{E}_x\left[f\left(Y_{t\wedge \sigma_{\infty}^{+}}\right)e^{-c\left(t\wedge \sigma_{\infty}^{+}\right)}\mathbf{1}_{\{\sigma_{\infty}^{+}<t\}}\right]\leq \mathbb{E}_x\left[f\left(Y_{t\wedge \sigma_{\infty}^{+}}\right)e^{-c\left(t\wedge \sigma_{\infty}^{+}\right)}\right]\leq f(x),$$
and we get a contradiction. We know then that $Y$ does not explode, and we can deduce that
\begin{equation} \label{eq proof thA (i)}
\mathbb{P}_x\left(\tau_{\infty}^+ = \infty \text{ or } \tau_{x_0}^- < \tau_{\infty}^+ < \infty\right) = 1.
\end{equation}
Suppose again by contradiction that $\mathbb{P}_x(\tau_{\infty}^+<\infty)>0$. Then for any $\lambda > 0$, $$\mathbb{E}_x\left[e^{-\lambda \tau_{\infty}^+},\, \tau_{\infty}^+ < \infty\right] > 0.$$
By applying the strong Markov property at the stopping time $\tau_{x_0}^-$, on the event $\{\tau_{x_0}^{-}<\infty\}$, we obtain 
\begin{align*}
\mathbb{E}_x\left[e^{-\lambda \tau_{\infty}^+},\, \tau_{x_0}^- < \tau_{\infty}^+ < \infty\right] &= \mathbb{E}_x\left[e^{-\lambda \tau_{x_0}^-}\mathbb{E}_{X_{\tau_{x_0}^{-}}}\big[e^{-\lambda \tau_{\infty}^+},\, \tau_{\infty}^+ < \infty\big],\, \tau_{x_0}^- < \infty \right] 
\\ &\leq \mathbb{E}_x\left[e^{-\lambda \tau_{x_0}^-},\, \tau_{x_0}^- < \infty\right] \mathbb{E}_{x_0}\left[e^{-\lambda \tau_{\infty}^+},\, \tau_{\infty}^+ < \infty\right],
\end{align*}
where we have used that a.s. $X_{\tau_{x_0}^{-}}\leq x_0$ and by monotonicity, $\mathbb{P}_{X_{\tau_{x_0}^{-}}}(\tau_\infty^+>t)\leq \mathbb{P}_{x_0}(\tau_\infty^+>t)$ for all $t\geq 0$. In a similar way, since $X_t^{x_0}\leq X_t^{x}$ a.s., we have
\begin{equation*}
\mathbb{E}_{x_0}\left[e^{-\lambda \tau_{\infty}^+},\, \tau_{\infty}^+ < \infty\right] \leq \mathbb{E}_x\left[e^{-\lambda \tau_{\infty}^+},\, \tau_{\infty}^+ < \infty\right].
\end{equation*}
Moreover, by \eqref{eq proof thA (i)},
\begin{equation*}
\mathbb{E}_{x}\left[e^{-\lambda \tau_{\infty}^+},\, \tau_{\infty}^+ < \infty\right] = \mathbb{E}_x\left[e^{-\lambda \tau_{\infty}^+},\, \tau_{x_0}^{-}<\tau_{\infty}^+ < \infty\right].
\end{equation*}
This leads to $\mathbb{E}_x\big[e^{-\lambda \tau_{x_0}^-}\big] \geq 1$, which contradicts the fact that $\mathbb{P}_x(\tau_{x_0}^- = 0) = 0$ for $x>x_0$. Hence, explosion is impossible.\\
For the cases where $x\leq x_0$, we have by comparison $X_t^{x}\leq X_t^{x_0+1}, \ \forall t\geq 0$, a.s.. We proved previously that $X^{x_0+1}$ is non-explosive. Thus $X^{x}$ is non-explosive and this concludes this point.

\paragraph{(ii)}

Let $f \in \mathcal{D}_\mathcal{X}$ be a nonnegative, non-decreasing function such that $f(x) \to \infty$ as $x \to \infty$, and assume that for all $x \geq x_0$, one has $\mathcal{X}f(x) \leq -c$ for some constant $c > 0$.

By assumption, the stopped process
\[
\left(f\!\left(X_{t \wedge \tau_{x_0}^{-}}\right) - \int_0^{t \wedge \tau_{x_0}^{-}} \mathcal{X}f(X_s)\,\ddr s\right)_{t \geq 0}
\]
is a positive local martingale, and hence a super-martingale. 

Applying the optional stopping theorem gives
\begin{align*}
\mathbb{E}_x\!\left[f\!\left(X_{t \wedge \tau_{x_0}^-}\right)\right] - f(x) 
&\leq \mathbb{E}_x\!\left[\int_0^{t \wedge \tau_{x_0}^-} \mathcal{X}f(X_s)\, \ddr s\right] \\
&\leq -c\,\mathbb{E}_x\!\left[t \wedge \tau_{x_0}^-\right].
\end{align*}
Since $f$ is non-decreasing and $X_{t \wedge \tau_{x_0}^-} \geq x_0$ a.s., we have $\mathbb{E}_x[f(X_{t \wedge \tau_{x_0}^-})] \geq f(x_0)$. Therefore,
\[
\mathbb{E}_x[\tau_{x_0}^-] 
\leq \lim_{t \to \infty} \mathbb{E}_x[t \wedge \tau_{x_0}^-] 
\leq \frac{1}{c}\big(f(x) - f(x_0)\big) < \infty.
\]

\paragraph{(iii)}

Let $f \in \mathcal{D}_\mathcal{X}$ be nonnegative, non-decreasing and bounded with $\mathcal{X}f(z) \leq -c$ for $z \geq x_0$.
Fix $x_0 \leq x \leq a$. Then, as in (ii):
\begin{equation*}
\mathbb{E}_x[f(X_{t \wedge \tau_a^-})] - f(x) \leq -c\mathbb{E}_x[t \wedge \tau_a^-].
\end{equation*}
Using that $f(X_{t \wedge \tau_a^-}) \geq f(a)$ and $f(x) \leq f(\infty)$, we deduce:
\begin{equation*}
\mathbb{E}_x[\tau_a^-] \leq \frac{1}{c}(f(\infty) - f(a)).
\end{equation*}
By letting $x \to \infty$ and using Lemma~\ref{lemma stopping times}, we get
\begin{equation*}
\mathbb{E}_{\infty}[\tau_a^-] \leq \frac{1}{c}(f(\infty) - f(a)).
\end{equation*}
Then letting $a \to \infty$, we find:
\begin{equation*}
\mathbb{E}_{\infty}[\tau_{\infty}^-] = 0.
\end{equation*}
Hence, $\tau_{\infty}^-= 0$, $\mathbb{P}_\infty$-a.s., which means that a.s. $X^{\infty}_t < \infty$ for all $t > 0$.
\end{proof}
\section{Proofs of the main results: coming down from infinity}
\label{section:proofs}

\subsection{Lyapunov functions for the CBDI}
\label{section:lyapfun}
%We study the CBDI using the extended generator $\mathcal{X}$ and apply the results from the preliminary section. We begin by defining the generator of the CBDI $X$ and formulating the associated local martingale problem. 
%From this point onward, $X$ denotes the CBDI($\Psi$,$I$), where $I:\mathbb{R}_{+} \mapsto \mathbb{R}$ satisfies [A]. 
%Before providing proofs of our results, we first describe/give the expression of the extended generator of $X$. In line with the second assumption made on general Markov processes introduced in the preliminary section, this  generator is the operator for which $X$ satisfies the martingale problem \eqref{eq pm}.
% \begin{lemma}
% \label{lemma generator}
% For $f$ function in $C^2$, $\left(f(X_t)-\int_0^t\mathcal{X}f(X_s)ds\right)_{t\geq 0}$ is a local martingale, with for all $z\in \mathbb{R}_{+}$
% \begin{equation}
% \label{generator}
% \mathcal{X}f(z)=-I(z)f'(z)-\gamma f'(z)+\frac{\sigma^2}{2}zf''(z)+z\int_0^{\infty}\left(f(z+u)-f(z)-uf'(z)\mathbf{1}_{\{u\leq 1\}}\right)\pi(du)
% \end{equation}
% We call $\mathcal{X}$ the extended generator of $X$. 
% \end{lemma}

% \begin{remark}
% If $f \in C_2([0,\infty))$ is bounded and $\underset{z>0}{\sup} \ \vert \mathcal{X}f(z) \vert<\infty$, then $\left(f(X_t)-\int_0^t\mathcal{X}f(X_s)ds\right)_{t\geq 0}$ is a martingale. 
% \end{remark}
%\paragraph*{}
 We proceed here to the proof of Theorem \ref{th descente inf}. Recall the expression in \eqref{generator} of the extended generator $\mathcal{X}$ of the CBDI. With the help of Condition [B], the assumptions of the theorem will enable us to define Lyapunov functions for $\mathcal{X}$. 
 %We have gathered some standard results about Lyapunov techniques in the Appendix; see in particular, Theorem \ref{th cond lyapunov}. 
 \vspace{2mm}
 
 We establish two lemmas, each defining a Lyapunov function based on the hypotheses of Theorem~\ref{th descente inf} : the first lemma addresses point (i) and establish non-explosion, while the second lemma tackles point (ii) and provides a sufficient condition for coming down from infinity.

\begin{lemma}
\label{lemma lyap 1}
Assume that $I$ is $C^1$ and that (B1) is satisfied. Then, let $f_1$ be a $C^2$ function such that, for any $z\geq \kappa$, $f_1(z)=\int_{\kappa}^z\frac{u}{I(u)}\ddr u$. One has 
\[\mathcal{X}f_1(z)=z\left(-1+\frac{\sigma^2}{2}\frac{I(z)-zI'(z)}{I(z)^2}+\epsilon_1(z)\right), \ z\geq \kappa,\]
with $\epsilon_1$ such that $\epsilon_1(z)\rightarrow 0$ as $z$ goes to $\infty$. In particular, there exist $M_1,c_1>0$ such that:
$$\forall z\geq M_1, \quad \mathcal{X}f_1(z)\leq -c_1.$$
\end{lemma}

\begin{proof}
Recall that assumption (B1) ensures
\[
\frac{z}{I(z)} \underset{z\rightarrow\infty}{\downarrow} 0.
\]
From \eqref{generator}, for $z \geq \kappa$ we obtain
\begin{align*}
\mathcal{X}f_1(z) 
&= -z + \frac{\sigma^2}{2} z f_1''(z) - \gamma zf_1'(z) 
   + z \int_0^{\infty} \Big( f_1(z+h) - f_1(z) - h f_1'(z)\mathbf{1}_{\{h \leq 1\}} \Big) \pi(\ddr h) \\
&= z\left(-1 + \frac{\sigma^2}{2}\frac{I(z)-zI'(z)}{I(z)^2} - \gamma \frac{z}{I(z)} 
   + \int_0^{\infty}\Big(f_1(z+h)-f_1(z)-hf_1'(z)\mathbf{1}_{\{h\leq 1\}}\Big)\pi(\ddr h)\right).
\end{align*}
We show that
\[
\epsilon_1(z):=-\gamma \frac{z}{I(z)} 
+ \int_0^{\infty} \Big(f_1(z+h)-f_1(z)-hf_1'(z)\mathbf{1}_{\{h \leq 1\}}\Big)\pi(\ddr h)
\;\underset{z \to \infty}{\longrightarrow}\; 0.
\]

It follows immediately from (B1) that
\[
-\gamma \frac{z}{I(z)} \;\underset{z \to \infty}{\longrightarrow}\; 0.
\]

For the integral over $(0,1]$, by Fubini’s theorem we have
\begin{align*}
\int_0^1 \Big(f_1(z+h)-f_1(z)-h f_1'(z)\Big)\pi(\ddr h)
&= \int_0^1 \left( \frac{u+z}{I(u+z)} - \frac{z}{I(z)} \right)\pi([u,1])\,\ddr u.
\end{align*}
Since $\tfrac{z}{I(z)} \downarrow 0$ as $z \to \infty$, we deduce that for all $u\in [0,1]$
\[
\left( \frac{u+z}{I(u+z)} - \frac{z}{I(z)} \right)\pi([u,1]) \;\underset{z\to\infty}{\longrightarrow}\; 0,
\]
and moreover
\[
\Big|\left( \tfrac{u+z}{I(u+z)} - \tfrac{z}{I(z)} \right)\pi([u,1])\Big|
\;\leq\; 2\frac{\kappa}{I(\kappa)}\pi([u,1]).
\]
The dominated convergence theorem then yields
\[
\int_0^1 \Big(f_1(z+h)-f_1(z)-h f_1'(z)\Big)\pi(\ddr h)
\;\underset{z\to\infty}{\longrightarrow}\; 0.
\]

For the integral over $(1,\infty)$, again by Fubini’s theorem,
\begin{align*}
\int_1^{\infty}\big(f_1(z+h)-f_1(z)\big)\pi(\ddr h)
&= \int_0^{\infty} \frac{u+z}{I(u+z)}\,\overline{\pi}(u \vee 1)\,\ddr u.
\end{align*}
Since $z \mapsto \tfrac{z}{I(z)}$ is nonincreasing, for all $u \geq 0$ we have
\begin{align*}
\frac{u+z}{I(u+z)}\,\overline{\pi}(u \vee 1)
&\leq \overline{\pi}(1)\frac{\kappa}{I(\kappa)}\mathbf{1}_{\{u \leq 1\}}
   + \frac{u\overline{\pi}(u)}{I(u)}\mathbf{1}_{\{u > 1\}} 
   \;=: g(u).
\end{align*}
The function $g$ is integrable on $(0,\infty)$. Moreover, for every fixed $u \geq 0$, we have
$\tfrac{u+z}{I(u+z)} \to 0$ as $z \to \infty$. By dominated convergence,
\[
\int_1^{\infty}\big(f_1(z+h)-f_1(z)\big)\pi(\ddr h) 
\;\underset{z\to\infty}{\longrightarrow}\; 0.
\]

Finally, by condition (B1), since $z \mapsto \tfrac{z}{I(z)}$ is nondecreasing on $[\kappa,\infty)$, we deduce that $f_1''(z) \leq 0$ for $z \geq \kappa$. Therefore,
\[
\mathcal{X}f_1(z) \;\leq\; z\big(-1 + \epsilon_1(z)\big) 
\;\underset{z\to\infty}{\longrightarrow}\; -\infty,
\]
which completes the proof.
\end{proof}

% \paragraph*{} Now that we have defined the first Lyapunov function $f_1$ (à reformuler), we note that it satisfies the assumptions of Theorem A, point (ii), when we assume the integral condition from Theorem \ref{th descente inf} (i). Applying Theorem A (ii) then guarantees that the CBDI does not explode. Since $f_1$ is positive on $(\kappa,\infty)$, Lemma \ref{lemma lyap 1} also implies Theorem A (i). Therefore $$\exists x_0\in (0,\infty), \ \forall x\in [x_0,\infty), \quad \mathbb{E}_x[\tau_{x_0}^{-}]<\infty,$$ and the proof of Theorem \ref{th descente inf} (i) is completed. By a similar argument, the proof of Theorem \ref{th descente inf} (ii) follows from the next lemma.

\begin{lemma}
\label{lemma lyap 2}
Let $f_2$ be a $C^2$ function such that, for any $z\geq \kappa$, $f_2(z)=\int_{\kappa}^z\frac{\ddr u}{I(u)}$. If $I$ satisfies (B1) and (B2), then 
\[\mathcal{X}f_2(z)=-1+\int_1^\infty\frac{z\bar{\pi}(u)}{I(u+z)}\ddr u+\epsilon_2(z), \ z\in [0,\infty)\]
where $\epsilon_2$ is such that $\epsilon_2(z)\rightarrow 0$ as $z$ goes to $\infty$. Furthermore, if $\mathcal{I}=\int_{\kappa}^{\infty}\frac{u\bar{\pi}(u)}{I(u)}\ddr u <\infty$, then
\begin{equation}
     \underset{z\to\infty}{\lim} \int_1^{\infty} \frac{z\overline{\pi}(u)}{I(u+z)} \ddr u= 0,
\end{equation}
and there exists $M_2,c_2>0$ such that:
$$\forall z\geq M_2, \quad \mathcal{X}f_2(z)\leq -c_2.$$
\end{lemma}

\begin{proof}
For \( z \geq \kappa \), the extended generator applied to \( f_2 \) reads:
\begin{align*}
\mathcal{X}f_2(z) &= -1 - \frac{\sigma^2}{2} z \frac{I'(z)}{I(z)^2} - \gamma \frac{z}{I(z)} + z \int_0^1 \left[ \int_z^{z+h} \left( \frac{1}{I(u)} - \frac{1}{I(z)} \right) \ddr u \right] \pi(\ddr h) + z \int_1^{\infty} \int_z^{z+h} \frac{\ddr u}{I(u)} \pi(\ddr h) \\
&= -1 - \frac{\sigma^2}{2} z \frac{I'(z)}{I(z)^2} - \gamma \frac{z}{I(z)} + z \int_0^1 \left( \frac{1}{I(u+z)} - \frac{1}{I(z)} \right) \pi([u,1]) \ddr u + z \int_0^{\infty} \frac{\overline{\pi}(u\vee 1)}{I(u+z)} \ddr u \\
% &= -1 - \frac{\sigma^2}{2} z \frac{I'(z)}{I(z)^2} - \gamma \frac{1}{I(z)} + z \int_0^1 \left( \frac{1}{I(u+z)} - \frac{1}{I(z)} \right) \pi([u,1]) \ddr u + \int_0^1 \frac{z\overline{\pi}(1)}{I(u+z)} \ddr u \\
% & \hspace{11cm} + \int_1^{\infty} \frac{z\overline{\pi}(u)}{I(u+z)} \ddr u\\
&= -1+\epsilon_2(z)+ \int_1^{\infty} \frac{z\overline{\pi}(u)}{I(u+z)} \ddr u,
\end{align*}
with
\[\epsilon_2(z):=- \frac{\sigma^2}{2} z \frac{I'(z)}{I(z)^2} - \gamma \frac{z}{I(z)} + z \int_0^1 \left( \frac{1}{I(u+z)} - \frac{1}{I(z)} \right) \pi([u,1]) \ddr u + \int_0^1 \frac{z\overline{\pi}(1)}{I(u+z)} \ddr u, \quad z\geq \kappa.\]
We aim to show that $\epsilon_2(z)$ goes to $0$ as $z$ goes to $\infty$.
%\( \mathcal{X}f_2(z) \to -1 \) as \( z \to \infty \), which will conclude the proof.
First, properties (B1) and (B2) imply:
\begin{align*}
z \frac{I'(z)}{I(z)^2} = \left( \frac{z}{I(z)} \right)^2 \frac{I'(z)}{z} \longrightarrow 0 \quad \text{and} \quad \frac{z}{I(z)} \longrightarrow 0 \quad \text{as} \quad z \to \infty.
\end{align*}
Moreover, under (B1) and (B2), there exist constants \( K_1, K_2 > 0 \) such that, for all \( z \geq \kappa \),
\begin{align*}
0 \leq \frac{z}{I(z)} \leq K_1, \quad \text{and} \quad 0 \leq \frac{I'(z)}{z} \leq K_2.
\end{align*}
Now, for \( z \geq \kappa \) and \( u \in [0,1] \), we have :
\begin{align*}
0 \leq z\left( \frac{1}{I(z)} - \frac{1}{I(u+z)} \right) \leq \frac{z}{I(z)} \leq K_1.
\end{align*}
Thus,
\[
z\left( \frac{1}{I(z)} - \frac{1}{I(u+z)} \right) \underset{z\to\infty}{\longrightarrow} 0.
\]
Furthermore, by the mean value theorem, for each \( u \in [0,1] \), there exists \( \delta_u \in (0,u) \) such that :
\begin{align*}
z\left( \frac{1}{I(z)} - \frac{1}{I(u+z)} \right) &= z u \frac{I'(z+\delta_u)}{I(z+\delta_u)^2} \\
&\leq z u \frac{K_2(z+\delta_u)}{I(z+\delta_u)^2} \\
&\leq K_2 u \left( \frac{z+\delta_u}{I(z+\delta_u)} \right)^2 \\
&\leq K_1^2 K_2 u.
\end{align*}
Since \(  u \in [0,1] \mapsto u \pi([u,1]) \) is integrable, because \(\int_0^1 h^2 \pi(\ddr h) < \infty\), it follows from the dominated convergence theorem that:
\[
z \int_0^1 \left( \frac{1}{I(u+z)} - \frac{1}{I(z)} \right) \pi([u,1]) \ddr u \underset{z\to\infty}{\longrightarrow} 0.
\]
Now, considering the second integral:
\begin{align*}
z \int_0^1 \frac{\overline{\pi}(1)}{I(u+z)} \ddr u \leq \frac{z\overline{\pi}(1)}{I(z+1)} = \overline{\pi}(1)\frac{z+1}{I(z+1)} \times \frac{z}{z+1},
\end{align*}
which tends to \( 0 \) as \( z \to \infty \) by (B1) and (B2). Gathering all terms and their limits, we conclude that the function $\epsilon_2$ vanishes at $\infty$. 
\smallskip

Assume now that $\mathcal{I}=\int^{\infty}\frac{u\bar{\pi}(u)}{I(u)}\ddr u<\infty$. Since
\[
\int_1^{\infty} \frac{z\overline{\pi}(u)}{I(u+z)} \ddr u \leq \int_1^{\infty} \frac{(u+z)\overline{\pi}(u)}{I(u+z)} \ddr u,
\]
and by assumption $u\mapsto I(u)/u$ is non-decreasing, we can bound the integrand of the second integral as follows:
$$\frac{u+z}{I(u+z)}\overline{\pi}(u)\leq \frac{u}{I(u)}\overline{\pi}(u).$$
By (B2), $\frac{u+z}{I(u+z)} \underset{z\rightarrow \infty}\rightarrow 0$ and we have by the dominated convergence theorem:
\[
\int_1^{\infty} \frac{z\overline{\pi}(u)}{I(u+z)} \ddr u \underset{z\to\infty}{\longrightarrow} 0.
\]
Therefore $\mathcal{X}f_2(z) \underset{z\to\infty}{\longrightarrow} -1$, which concludes the proof.
\end{proof}
\noindent\textit{Proof of Theorem \ref{th descente inf}.}
\begin{enumerate}
    \item[i)] Under the assumption $\int^{\infty}\frac{u\bar{\pi}(u)}{I(u)}du<\infty$,  Lemma \ref{lemma lyap 1} provides a positive non-decreasing function $f_1$ such that $f_1(x)\rightarrow \infty$ as $x$ goes to $\infty$ and $\mathcal{X}f_1(x)\leq -c_1\leq f_1(x)$ for $x$ large enough. By Theorem \ref{th cond lyapunov}-(i) and (ii), the process $X$ does not explode and the first entrance times away from $0$ are integrable. Namely $$\exists x_0\in (0,\infty), \ \forall x\in [x_0,\infty), \quad \mathbb{E}_x[\tau_{x_0}^{-}]<\infty.$$ The proof of Theorem \ref{th descente inf} (i) is completed.
    \item[ii)] Under the assumption  $\int^{\infty}\frac{du}{I(u)}<\infty$, the function $f_2$ in Lemma \ref{lemma lyap 2} is \textit{bounded}, still with the condition $\int^{\infty}\frac{u\bar{\pi}(u)}{I(u)}du<\infty$ in force, satisfies $\mathcal{X}f_2\leq -c_2$ in a neighborhood of $\infty$. By applying Theorem \ref{th cond lyapunov}-(iii), we see that the process comes down from infinity and satisfies  
    $$\exists x_0\in (0,\infty), \quad \mathbb{E}_\infty[\tau_{x_0}^{-}]<\infty.$$ \qed
\end{enumerate}

\subsection{Regularity of the CBDI in the initial state}
\label{section:regularity}
\paragraph*{}
Up to this point, we have considered the process starting from infinity as the pointwise limit along the initial value. However, without additional properties - such as the Feller property of the CBDI on $[0,\infty)$ (see e.g. Foucart et al. \cite{FLi_entrance}) - this approach does not seem sufficient to establish that the process started from $\infty$ has càdlàg sample paths and satisfies the strong Markov property. We will show that, under an additional condition on the drift function $I$, the process $\left(X^{\infty}_t, \  t\geq 0\right)$ is the unique solution of a certain stochastic equation. The representation will directly imply the desired properties.
\paragraph*{}
We first present a result about the regularity of the process $X$ on its initial values. This will turn to be crucial in our study. 
%Proposition \ref{prop feller} which concerns 
The key additional assumption on $I$ is the one-sided Lipschitz condition.

\begin{proposition}
\label{prop feller}
Assume that $X$ is not explosive and that $I$ satisfies (B3).
Then, for any $t\geq 0$, $y\in [0,\infty)$ and $(y_n)_{n\geq 1}\in \mathbb{R}^{\mathbb{N}}$ such that $y_n\rightarrow y$ when $n\rightarrow\infty$, we have:
\begin{equation}
\mathbb{P}\big(\underset{n \rightarrow \infty}{\lim} \ \underset{s\leq t}{\sup} \ \vert X^{y_n}_s-X^y_s\vert= 0\big)=1.
\end{equation}
\end{proposition}

% \begin{remark}
% Under the hypothesis of Proposition \ref{prop feller},
% the semigroup  $(P_t)_{t\geq 0}$ of $X$, satisfies the following Feller properties, where $C$ is the space of continuous function from $[0,\infty)$ to $[0,\infty)$ :
% \begin{equation}
% \left\{
% \begin{array}{ll}
% &\forall f \in C([0,\infty]), \ P_tf \text{ is continuous on } [0,\infty). \\
% &\forall f\in C([0,\infty]), \ \forall x \in [0,\infty), \ P_tf(x) \underset{t\rightarrow 0}{\rightarrow} f(x)
% \end{array}
% \right.
% \end{equation}
% \end{remark}

\begin{proof}
We will need here a construction of the CBDI process with all initial values. Consider the stochastic equation \eqref{eds} where the term driven by the Brownian motion $B$ is replaced by the term \eqref{whitenoise}  with a white noise $W$. Call $\big(X_t(x),t\geq 0, x\in [0,\infty)\big)$ a family of solutions for which $X_0(x)=x$ for all $x\in [0,\infty)$. The latter uniquely exists, is càdlàg in $t$ and is equivalent for every fixed $x\in [0,\infty)$, to $(X^x_t,t\geq 0)$. We refer the reader to \cite{DawsonLi}. We set, for $x,y \in (0,\infty)$, $D_t^{x,y}=X_t(x)-X_t(y)$. We are going to show that  \[\forall t\geq 0, \ \forall x,y\in (0,\infty), \quad \underset{s\leq t}{\sup} \ D_s^{x,y} \ \rightarrow 0 \quad \textup{a.s. when} \ x \uparrow y.\] For $y<x$ and $t\geq 0$, we have:
\begin{multline}
\label{eds D}
D_t^{x,y}=x-y+\sigma\int_0^t\int_0^{D^{x,y}_{s-}}W'(\ddr s,\ddr u)-\gamma\int_0^tD_s^{x,y}\ddr s-\int_0^t\left[I(X_s(x))-I(X_s(y))\right]\ddr s \\ +\int_0^t\int_0^{D_{s-}^{x,y}}\int_0^1h\tilde{M'}(\ddr s,\ddr u,\ddr h)+\int_0^t\int_0^{D_{s-}^{x,y}}\int_1^{\infty}hM'(\ddr s,\ddr u,\ddr h)
\end{multline}
with $W'(\ddr s,\ddr u)= W(\ddr s,\ddr u+X_{s-}(y))$ and $M'(\ddr s,\ddr u,\ddr z)=M(\ddr s,\ddr u+X_{s-}(y)),\ddr z)$. By applying the same arguments as in the proof of Proposition 1.1 in \cite[Page 15-16]{Berestycki}, we deduce that $W'$ is a Gaussian white noise with same intensity as $W$, that $M'$ is a Poisson random measure with same intensity as $M$ and that $M'$ and $W'$ are independent. 
\vspace{2mm}

Recall the hypothesis (B3) on $I$: \[ \exists b>0, \ \forall y,z\geq 0, \quad I(y+z)-I(y)\geq -bz .\] We take such $b$ for the rest of the proof. This leads to the following comparison result for the process $D^{x,y}$, which we admit for the moment and whose proof will be provided in the forthcoming Lemma~\ref{lemma comparison D}:
\begin{equation}
\label{comparison D}
\forall t\geq 0, \quad D_t^{x,y}\leq Y_t(x-y) \quad \textup{a.s.}
\end{equation}
where $Y(x-y)$ satisfies the stochastic equation:
\begin{multline}
Y_t=x-y+\sigma\int_0^t\int_0^{Y_{s-}}W'(\ddr s,\ddr u)+(b-\gamma)\int_0^tY_s\ddr s \\ +\int_0^t\int_0^{Y_{s-}}\int_0^1h\tilde{M'}(\ddr s,\ddr u,\ddr h)
+\int_0^t\int_0^{Y_{s-}}\int_1^{\infty}hM'(\ddr s,\ddr u,\ddr h).
\end{multline}
Observe that the process $Y$ is a CB($\hat{\Psi}$) with $\hat{\Psi}(x):=\Psi(x)-bx$ where $\Psi$ is the branching mechanism of the CB related to $X$, see \eqref{eq LK} (with $\lambda=0$). 

Granting \eqref{comparison D}, the proof follows from an application of Lemma 1.5 in Duquesne and Labbé \cite[page 8]{DuquesneLabbe}. The latter yields the following convergence in probability, for a fixed $t\geq 0$, one has
\[\underset{r\rightarrow 0}{\lim }\ \underset{s\leq t}{\sup} \, Y_s(r)=0.\] In our setting, since $r\mapsto \underset{s\leq t}{\sup} \, Y_s(r)$ is a.s. non-decreasing, we deduce that the convergence actually holds almost surely. Consequently, we obtain :
\begin{equation*}
0 \leq \underset{s\leq t}{\sup} \, \vert X_t(x)-X_t(y) \vert = \underset{s\leq t}{\sup} \, D_s^{x,y} \leq \underset{s\leq t}{\sup} \, Y_s(x-y) \ \underset{x\downarrow y}{\longrightarrow} \ 0 \quad \textup{a.s.},
\end{equation*}
which completes the proof.
\end{proof}

% \paragraph*{}
% We will appeal next the following result on classical CB processes obtained by Duquesne and Labb\'e \cite{DuquesneLabbe}.  
% %which establishes the convergence in probability of a continuous-state branching process and of its running supremum.
% \begin{lemma}[Lemma 1.5 in \cite{DuquesneLabbe}]\label{lemma D-L} Let $\hat{\psi}$ be a branching mechanism (with general form \eqref{eq LK}) and $Z$ be a càdlàg CB($\hat{\psi}$). Then, for all $(a,y) \in (0,\infty)^2$ :
% \begin{equation}
% \underset{r\rightarrow 0^{+}}{\lim} \ \mathbb{P}_r(Z_a>y)=0 \quad \textup{and} \quad \underset{r\rightarrow 0^{+}}{\lim} \ \mathbb{P}_r\Big(\underset{b\in [0,a]}{\sup} \ Z_b>y\Big)=0
% \end{equation}
% \end{lemma}

\paragraph{} We now turn in the next lemma to the proof of the comparison result \eqref{comparison D}. We continue to work within the framework of Proposition \ref{prop feller}, under the same assumptions and with the same processes $D^{x,y}$ and $Y$ introduced there. The proof adapts the arguments of Theorem~2.2 in \cite{DawsonLi}. A difficulty arises from the fact that, unlike in the setting of \cite{DawsonLi}, the CBDI - and hence the process $D^{x,y}$ -
may not have finite expectation. To overcome this, we modify the process by truncating all jumps above a fixed level $n$, replacing them by jumps of size $n$. This yields the processes $D^{n}$ and $Y^n$ constructed in the proof below, for which the expectations are finite. In this framework, the argument for establishing the comparison property in \cite{DawsonLi} can be applied.

\begin{lemma}
\label{lemma comparison D} The inequality \eqref{comparison D} holds true, i.e. we have 
\begin{align*}
\forall t\geq 0, \quad D_t^{x,y}\leq Y_t(x-y) \quad \textup{a.s.}
\end{align*}
\end{lemma}

\begin{proof}
Fix an integer \( n \geq 1 \), and let \( J_n \) denote the first time at which the process \( D^{x,y} \) undergoes a jump of size at least \( n \), that is,
\[
J_n := \inf\{ t \geq 0 \, : \, D_t^{x,y} - D_{t-}^{x,y} \geq n \}.
\]
Clearly, \( J_n > 0 \) almost surely. Define \( (D_t^n)_{t<J_n} :=(D_t^{x,y})_{t<J_n} \). By construction, the process \( D^n \) satisfies the same stochastic equation as \( D^{x,y} \), except that jumps larger than \( n \) are truncated at size \( n \). More precisely, \( D^n \) solves the following equation for \( t<J_n \):
\begin{multline}
D_t^n = x-y + \sigma \int_0^t \int_0^{D_{s-}^n} W'(\ddr s,\ddr u) -\gamma\int_0^tD_s^n\ddr s - \int_0^t \left[I(X_s(y)) - I(X_s(x))\right] \ddr s \\
+ \int_0^t \int_0^{D_{s-}^n} \int_0^1 h \, \tilde{M}'(\ddr s,\ddr u,\ddr h) + \int_0^t \int_0^{D_{s-}^n} \int_1^\infty (h \wedge n) M'(\ddr s,\ddr u,\ddr h).
\end{multline}
Fix \( t \geq 0 \). We aim to show that \(\forall t<J_n, \  D_t^n \leq Y_t^n \) a.s., where \( Y^n \) denotes the unique strong solution to the stochastic equation:
\begin{multline}
Y^n_t=x-y+\sigma\int_0^t\int_0^{Y^n_{s-}}W'(\ddr s,\ddr u)+(b-\gamma)\int_0^tY^n_s\ddr s\\+\int_0^t\int_0^{Y^n_{s-}}\int_0^1h\tilde{M'}(\ddr s,\ddr u,\ddr h)+\int_0^t\int_0^{Y^n_{s-}}\int_1^{\infty}(h\wedge n)M'(\ddr s,\ddr u,\ddr h).
\end{multline}
We now introduce the process \( \zeta^n \) defined by
\[
\forall t\geq 0, \quad \zeta^n_t := D_t^n - Y_t^n.
\]
Our goal is to establish that, \(\forall t<J_n, \ \zeta_t^n \leq 0 \) a.s.. To this end, we will show that
\[
\mathbb{E}\left[\left(\zeta_t^n\right)^{+} \mathbf{1}_{\{t<J_n\}}\right] = 0,
\]
where for any \( z\in\mathbb{R} \), we set \( z^{+} := z\vee 0 \).

%To apply Itô's formula,  we proceed via 
We will need a smooth approximation of the positive part function \( z\mapsto z^{+} \). Let \( (\phi_k)_{k\geq 1} \) be a sequence of \( C^2 \) functions satisfying the following properties: 
\begin{itemize}
\item[•] For all \( z\leq 0 \), \( \phi_k(z) = \phi_k'(z) = 0 \).
\item[•] For all \( z\in\mathbb{R} \), \( \phi_k(z) \to z^{+} \) as \( k\to\infty \).
\item[•] For all $z\in \mathbb{R}$, \( \phi_k'(z)\leq 1 \), \( \phi_k''(z)\leq \frac{2}{kz} \).
\end{itemize}
The construction of such a sequence is recalled in the Appendix.

\vspace{0.3cm}

Define for all $t\geq 0$, \( Z_t := I(X_t(y)) - I(X_t(x)) \). By the one-sided Lipschitz assumption (B3), we have
\begin{equation}
\label{eq B3 preuve comp}
\forall t<J_n, \quad  Z_t \leq b D_t^{x,y} = b D_t^n \quad \text{a.s.}.
\end{equation}

Finally, we note the following important property, which will be used later: on the event \( \{\zeta_{s-}^n\leq 0\} \), the following holds almost surely:
\begin{equation}
\label{useful point}
\phi_k(\zeta^n_{s-})=\phi_k'(\zeta^n_{s-})=\phi_k''(\zeta^n_{s-})=0 \quad \textup{and} \quad \zeta^n_{s-}+\mathbf{1}_{\{z\leq D_{s-}^{n}\}}h-\mathbf{1}_{\{z\leq Y_{s-}^n\}} h \leq 0 \quad \textup{a.s.}.
\end{equation} 

Our aim is to analyze the expectation $\mathbb{E}\left[\phi_k\left(\zeta^n_{t\wedge\tau_m}\right)\mathbf{1}_{\{t<J_n\}}\right]$ and then let $k\rightarrow\infty$ in order to establish the inequality: $$\mathbb{E}\left[\left(\zeta^n_{t\wedge\tau_m}\right)^{+}\mathbf{1}_{\{t<J_n\}}\right]\leq c\int_0^t\mathbb{E}\left[\left(\zeta^n_{s\wedge\tau_m}\right)^{+}\mathbf{1}_{\{s\leq t<J_n\}}\right] \ddr s,$$
for some constant $c>0$. To this end, we apply Itô's formula to the semimartingale $(\zeta^n_t, t\geq 0)$ using the function \( \phi_k \). By applying  \cite[Theorem 93, page 59]{Situ}, we have
\begin{equation}
\label{eq ito on diff}
\begin{split}
\phi_k(\zeta^n_t)&=\phi_k(\zeta_0)+\frac{\sigma^2}{2}\int_0^t\mathbf{1}_{\{\zeta^n_{s-}>0\}}\ddr s\int_0^{\infty}\phi_k''(\zeta^n_{s-})\left(\mathbf{1}_{\{u\leq D_{s-}^{n}\}}-\mathbf{1}_{\{u\leq Y_{s-}^{n}\}}\right)^2\ddr u\\
& \quad + \int_0^t\phi_k'(\zeta^n_{s-})(Z_{s-}-bY^n_{s-}-\gamma\zeta^n_{s-})\mathbf{1}_{\{\zeta^n_{s-}>0\}} \ \ddr s \\
& \quad + \int_0^t\ddr s\int_0^{\infty}\int_1^{\infty}[\phi_k\left(\zeta^n_{s-}+(\mathbf{1}_{\{u\leq D_{s-}^{n}\}}-\mathbf{1}_{\{u\leq Y^n_{s-}\}})(h\wedge n)\right)-\phi_k(\zeta^n_{s-})]\mathbf{1}_{\{\zeta^n_{s-}>0\}} \ \ddr u\pi(\ddr h)\\
& \quad + \int_0^t\ddr s\int_0^{\infty}\int_0^{1}\Big[\phi_k\left(\zeta^n_{s-}+(\mathbf{1}_{\{u\leq D_{s-}^{n}\}}-\mathbf{1}_{\{u\leq Y^n_{s-}\}})(h\wedge n)\right)-\phi_k(\zeta^n_{s-})\\
& \hspace{6 cm} -(\mathbf{1}_{\{u\leq D^n_{s-}\}}-\mathbf{1}_{\{u\leq Y^n_{s-}\}})(h\wedge n)\phi_k'(\zeta^n_{s-})\Big]\mathbf{1}_{\{\zeta^n_{s-}>0\}}\ddr u \pi(\ddr h) \\  
& \quad +  M^n(t),
\end{split}
\end{equation}
where $M^n(t)$ is the local martingale given by:
\begin{align*}
M^n(t)&= \sigma \int_0^t\int_0^{\infty}\phi_k'(\zeta^n_{s-})\left(\mathbf{1}_{\{u\leq D_{s-}^n\}}-\mathbf{1}_{\{u\leq Y_{s-}^n\}}\right)W'(\ddr s,\ddr u) \\
& \quad + \int_0^t\int_0^{\infty}\int_0^{\infty}\left[\phi_k(\zeta^n_{s-}+(\mathbf{1}_{\{u\leq D^n_{s-}\}}-\mathbf{1}_{\{u\leq Y^n_{s-}\}})(h\wedge n))-\phi_k(\zeta^n_{s-})\right]\tilde{M}'(\ddr s,\ddr u,\ddr h).
\end{align*} 
For any $m\geq 0$, set \[
\tau_m := \inf\{ t \geq 0 \,:\, D_t^n \geq m \ \text{or} \ Y_t^n \geq m \}.
\]
As we shall need it later, we show now that \( M^n(\cdot \wedge \tau_m) \) is an \( L^2 \)-martingale. 
%where we apply standard results from semimartingale theory using the bracket \( [\cdot,\cdot] \), which denotes the quadratic variation. %distinct from the predictable quadratic variation (ou plutôt continuous quadratic variation ?) \( \langle \cdot,\cdot \rangle \), as it takes jumps into account.
By Corollary 3, p.73 in \cite{ProtterBook}, it suffices to show that
\[
\mathbb{E}\Big[ \big[ M^n(\cdot \wedge \tau_m), M^n(\cdot \wedge \tau_m) \big]_t \Big] < \infty.
\]
%which implies that \( M^n(\cdot \wedge \tau_m) \) is an \( L^2 \)-martingale.
We decompose \( M^n \) as follows:
\begin{align*}
M^{1,n}(t) &= \sigma \int_0^t \int_0^{\infty} \phi_k'(\zeta^n_{s-}) \left( \mathbf{1}_{\{u \leq D_{s-}^n\}} - \mathbf{1}_{\{u \leq Y_{s-}^n\}} \right) W'(\ddr s,\ddr u), \\
M^{2,n}(t) &= \int_0^t \int_0^{\infty} \int_0^{\infty} \left[ \phi_k\left(\zeta^n_{s-} + (\mathbf{1}_{\{u \leq D_{s-}^n\}} - \mathbf{1}_{\{u \leq Y_{s-}^n\}})(h \wedge n)\right) - \phi_k(\zeta^n_{s-}) \right] \tilde{M}'(\ddr s,\ddr u,\ddr h).
\end{align*}
By independence of \( W' \) and \( M' \), the cross variation \( [M^{1,n}, M^{2,n}] \) is identically zero. Since \( M^{1,n} \) is a continuous local martingale and \( \phi_k' \leq 1 \), one has
\begin{align*}
\left[ M^{1,n}(\cdot \wedge \tau_m), M^{1,n}(\cdot \wedge \tau_m) \right]_t 
&= \sigma^2 \int_0^{t \wedge \tau_m} \int_0^{\infty} \phi_k'^2(\zeta^n_{s-}) \left( \mathbf{1}_{\{u \leq D_{s-}^n\}} - \mathbf{1}_{\{u \leq Y_{s-}^n\}} \right)^2 \ddr u \ddr s \\
&\leq \sigma^2 \int_0^{t \wedge \tau_m} |\zeta^n_s| \, \ddr s \leq m \sigma^2 t.
\end{align*}

The process \( M^{2,n} \) is purely discontinuous (in the sense of Protter's definition of purely jump martingales in \cite[page 193]{ProtterBook}), so its quadratic variation is given by the sum of the squares of its jumps:
\begin{align*}
\left[ M^{2,n}(\cdot \wedge \tau_m), M^{2,n}(\cdot \wedge \tau_m) \right]_t 
= \int_0^{t \wedge \tau_m} \int_0^{\infty} \int_0^{\infty} 
\left( \phi_k(\zeta^n_{s-} + \Delta_{s,u,h}) - \phi_k(\zeta^n_{s-}) \right)^2 M'(\ddr s,\ddr u,\ddr h),
\end{align*}
where \( \Delta_{s,z,h} := (\mathbf{1}_{\{u \leq D_{s-}^n\}} - \mathbf{1}_{\{u \leq Y_{s-}^n\}})(h \wedge n) \). The expectation formula for integrals of positive predictable processes with respect to a Poisson random measure provides:
\begin{align*}
\mathbb{E}&\Big[ \big[ M^{2,n}(\cdot \wedge \tau_m), M^{2,n}(\cdot \wedge \tau_m) \big]_t \Big]\\
&= \int_0^t \int_0^{\infty} \int_0^{\infty} \mathbb{E}\left[ 
\left( \phi_k(\zeta^n_{s-} + \Delta_{s,u,h}) - \phi_k(\zeta^n_{s-}) \right)^2 
\mathbf{1}_{\{s < \tau_m\}} \right] \ddr s\, \ddr u\, \pi(\ddr h).
\end{align*}
By the mean value theorem and \( \phi_k' \leq 1 \), we have for all \( s < \tau_m \):
\begin{align*}
\left( \phi_k(\zeta^n_{s-} + \Delta_{s,z,h}) - \phi_k(\zeta^n_{s-}) \right)^2 
\leq \left( \Delta_{s,z,h} \right)^2 
= \left( \mathbf{1}_{\{z \leq D_{s-}^n\}} - \mathbf{1}_{\{z \leq Y_{s-}^n\}} \right)^2 (h \wedge n)^2.
\end{align*}
By Fubini–Tonelli’s theorem and \( \sigma \)-finiteness of \( \pi \), we conclude:
\begin{align*}
\mathbb{E}\Big[ \big[ M^{2,n}(\cdot \wedge \tau_m), M^{2,n}(\cdot \wedge \tau_m) \big]_t \Big] 
\leq mt \left( n^2 \int_1^{\infty} \pi(\ddr h) + \int_0^1 h^2 \pi(\ddr h) \right).
\end{align*}
Combining both parts, we find
%\[
$\mathbb{E}\Big[ \big[ M^n(\cdot \wedge \tau_m), M^n(\cdot \wedge \tau_m) \big]_t \Big] < \infty$,
%\]
and therefore \( M^n(\cdot \wedge \tau_m) \) is an \( L^2 \)-martingale.
\\

We now analyze the limit of  
\[
\mathbb{E}\left[\phi_k\left(\zeta^n_{t\wedge\tau_m}\right)\mathbf{1}_{\{t<J_n\}}\right] \quad \text{as } k \to \infty,
\]  
we examine each term in the Itô formula expansion \eqref{eq ito on diff}.

\begin{itemize}
    \item \textit{Term involving \( \phi_k'' \)} :  
    For all \( z \geq 0 \), we have \( \phi_k''(z) \leq \frac{2}{kz} \), so that:
    \begin{align*}
    &\mathbb{E}\left[\int_0^{t\wedge \tau_m} \mathbf{1}_{\{\zeta^n_{s-}>0\}} \, \ddr s \int_0^{\infty} \phi_k''(\zeta^n_{s-}) \left( \mathbf{1}_{\{u \leq D_{s-}^n\}} - \mathbf{1}_{\{u \leq Y_{s-}^n\}} \right)^2 \ddr u \right] \\
    &\quad \leq \frac{2}{k} \mathbb{E}\left[\int_0^{t\wedge \tau_m} \mathbf{1}_{\{\zeta^n_{s-}>0\}} \frac{1}{\zeta^n_{s-}} \zeta^n_{s-} \, \ddr s \right] = \frac{2}{k} \mathbb{E}[t \wedge \tau_m] \xrightarrow[k \to \infty]{} 0.
    \end{align*}

    \item \textit{Drift term}:  
    Since \(\forall s < J_n, \  Z_s \leq b D_s^n \) a.s. (see \eqref{eq B3 preuve comp}), we deduce that:
    \[\forall s < J_n, \quad
    \phi_k'(\zeta^n_{s-})(Z_{s-} - bY^n_{s-} - \gamma \zeta^n_{s-}) \leq (b - \gamma) \zeta^n_{s-} \quad \text{a.s.}.
    \]

    \item \textit{Large jump term (first-order estimate)}:  
    By the mean value theorem, on \( \{\zeta^n_{s-}>0\} \) and using \( \phi_k'(z) \leq 1 \), we have:
    \[
    \phi_k(\zeta^n_{s-} + \Delta_{s,z,h}) - \phi_k(\zeta^n_{s-}) \leq \Delta_{s,z,h} \quad \textup{a.s.}.
    \]

    \item \textit{Small jump term (second-order estimate)}:  
    Again by the mean value theorem and the bound \( \phi_k''(z) \leq \frac{2}{kz} \), we obtain:
    \[
    \phi_k(\zeta^n_{s-} + \Delta_{s,z,h}) - \phi_k(\zeta^n_{s-}) - \Delta_{s,z,h} \phi_k'(\zeta^n_{s-}) \leq \frac{1}{k\zeta^n_{s-}} \Delta_{s,z,h}^2.
    \]
\end{itemize}

Taking expectations in \eqref{eq ito on diff} and applying the above estimates yields:
\begin{align*}
\mathbb{E}\left[\phi_k\left(\zeta^n_{t\wedge \tau_m}\right)\mathbf{1}_{\{t<J_n\}}\right]
&\leq \frac{\sigma^2}{2} \mathbb{E}\left[\int_0^{t\wedge \tau_m} \mathbf{1}_{\{\zeta^n_{s-}>0\}} \, \ddr s \int_0^{\infty} \phi_k''(\zeta^n_{s-}) \left( \mathbf{1}_{\{u \leq D_{s-}^n\}} - \mathbf{1}_{\{u \leq Y_{s-}^n\}} \right)^2 \ddr u \right] \\
&\quad + (b - \gamma)\mathbb{E}\left[\mathbf{1}_{\{t<J_n\}}\int_0^{t\wedge \tau_m} \zeta^n_{s-} \mathbf{1}_{\{\zeta^n_{s-}>0\}} \ddr s\right] \\
&\quad + n\left(\int_1^\infty \pi(\ddr h)\right)\mathbb{E}\left[\mathbf{1}_{\{t<J_n\}}\int_0^{t\wedge \tau_m} \zeta^n_{s-} \mathbf{1}_{\{\zeta^n_{s-}>0\}} \ddr s\right] \\
&\quad + \frac{1}{k} \left(\int_0^1 h^2 \pi(\ddr h)\right) \mathbb{E}\left[\int_0^{t\wedge \tau_m} \mathbf{1}_{\{\zeta^n_{s-}>0\}} \ddr s\right],
\end{align*}
where the first and last terms vanish as \( k \to \infty \).

Letting \( k \to \infty \) and applying the monotone convergence theorem, we deduce:
\[
\mathbb{E}\left[\left(\zeta^n_{t\wedge \tau_m}\right)^+ \mathbf{1}_{\{t<J_n\}}\right] \leq c \int_0^t \mathbb{E}\left[\left(\zeta^n_{s\wedge \tau_m}\right)^+ \mathbf{1}_{\{s \leq t < J_n\}}\right] \ddr s,
\]
where \( c := b - \gamma + n\int_1^\infty \pi(\ddr h) > 0 \) for \( n \) large enough, since $\bar{\pi}(1)>0$ by assumption. An application of Gr\"onwall’s lemma yields:
\[
\mathbb{E}\left[\left(\zeta^n_{t\wedge \tau_m}\right)^+ \mathbf{1}_{\{t<J_n\}}\right] = 0,
\]
so that
%\[
$\left(\zeta^n_{t\wedge \tau_m}\right)^+ = 0 \ \text{ on } \{t < J_n\} \quad  \text{a.s.}.$
%\]
which in turn ensures that \( D_t^{x,y} = D_t^n \leq Y_t^n \) for all \( t\in [0, J_n) \) almost surely.  

Since furthermore the process \( X \) does not explode, we know that
\(
J_n \rightarrow \infty
\) a.s. when $n \rightarrow\infty$.
Hence, the inequality \( D_t^{x,y} \leq Y_t^n \) holds a.s. for all \( t \geq 0 \). Finally, as \( Y_t^n \uparrow Y_t \) almost surely as \( n \to \infty \), we obtain the desired result. 
\end{proof}

\subsection{The stochastic equation satisfied by $(X_t^{\infty}, \ t> 0)$}
\label{section:edsinf}

\paragraph*{}
This section is dedicated to the proofs of Proposition \ref{prop cv unif} and Theorem \ref{th descente eds}. In all this section we assume that $I$ satisfies the one-sided Lipschitz condition (B3), i.e. that:  \[ \exists b>0, \ \forall y,z\geq 0, \quad I(y+z)-I(y)\geq -bz, \] and that $\tau_a^{\infty,-}<\infty$ a.s. for some $a\in (0,\infty)$.
We start by establishing Proposition \ref{prop cv unif}, this will be a consequence of Proposition \ref{prop feller} and Lemma \ref{lemma comparison D}.
\\\\
\textbf{Proof of Proposition \ref{prop cv unif}.} Since $y_n, y \in [0,\infty)$ for all $n \in \mathbb{N}$, we have 
\[
\lvert e^{-y_n} - e^{-y} \rvert \leq \lvert y_n - y \rvert,
\] 
and the local uniform convergence follows directly from Proposition~\ref{prop feller}.  

For the second part, assume that $I(0)=0$ and use the one-sided Lipschitz condition (B3). For any $z \geq 0$,
\[
I(z) \geq I(0) - b z.
\]
By the comparison property (CP1), this implies that for all $n \in \mathbb{N}$ and $t \geq 0$, 
\[
X_t^{x_n} \leq Y_t^n \qquad \text{a.s.,}
\]
where $Y^n$ denotes a CB started from $x_n$ with branching mechanism $\tilde{\Psi}(x) := \Psi(x) + b x$.  

From Lemma~1.5 in \cite[page 8]{DuquesneLabbe}, we know that for all $t \geq 0$, 
\[
Y_t^{x_n} \xrightarrow[n \to \infty]{} 0 \quad \text{in probability}.
\] 
Consequently, $X_t^{x_n} \to 0$ in probability as $n \to \infty$. Since the sequence $\left(X_t^{x_n}\right)_{n \geq 0}$ is non-increasing a.s., the convergence in fact holds almost surely, which completes the proof.\qed

\paragraph*{}
% To establish $X^{\infty}$ as the unique strong solution of a stochastic equation, we will demonstrate the identity in distribution between the processes $\big(\tilde{X}^n_{t-\tau_n^{-}(\infty)}(n), \ t\geq \tau_n^{-}(\infty)\big)$ and $\big(X^{\infty}_t, \ t\geq \tau_n^{-}(\infty)\big)$, where $\tilde{X}^y$ denotes the solution of a shifted stochastic equation for a fixed $y\in (0,\infty)$. This identity will explicitely give the stochastic equation satisfied by $X^{\infty}$. Naturally, the regularity property of Proposition  \ref{prop feller} also applies to the shifted process $\tilde{X}^y$. This regularity result, formulated in Lemma~\ref{lemma feller tilde}, provides the key arguments to prove the convergence of $\left(X^x_t, \ s\leq t\right)$ towards $\left(X^{\infty}_t, \ s\leq t\right)$ in $(D,\rho_{\infty})$ as $x\rightarrow \infty$.
To establish that $X^{\infty}$ is the unique strong solution of a stochastic equation, we show that the processes
\[
\big(X^{\infty}_t, \ t\geq \tau_n^{\infty,-}\big)
\quad \text{and} \quad 
 \big(\tilde{X}^n_{t-\tau_n^{-}(\infty)}(n), \ t\geq \tau_n^{\infty,-}\big) 
\]  
have the same distribution, where $\tilde{X}^n$ denotes the solution of a shifted stochastic equation for fixed $n \in (0,\infty)$.
This identity in law explicitly characterizes the stochastic equation satisfied by $X^{\infty}$.
Moreover, the regularity property of Proposition~\ref{prop feller} also applies to the shifted process $\tilde{X}^n$.
The corresponding result, stated in Lemma~\ref{lemma feller tilde}, provides the key tool to prove the convergence of 
\[
\big(X^x_t, \ s\leq t\big)  \text{ towards }\big(X^{\infty}_t, \ s\leq t\big) \text{ in } (D,\rho_{\infty}) \text{ as } x \to \infty.\]

\paragraph*{}

The following proposition introduces the process called $\tilde{X}^y$ as the unique strong solution of the shifted stochastic equation \eqref{eds shift}. Since by hypothesis $\tau_a^{\infty,-}<\infty$ a.s., we have $\tau_y^{\infty,-}<\infty$ a.s. for any $y\in (a,\infty)$.

\begin{proposition} 
\label{prop existence tilde}
For $y \in (a,\infty)$ and $\xi$ a positive random variable, there exists a unique càdlàg strong Markov process, which we denote by $\left(\tilde{X}_t^{y}(\xi), \ t\geq 0\right)$, solution of the stochastic equation:
\begin{multline}
\label{eds shift}
X_t=\xi+ \sigma  \int_0^t\sqrt{X_s}\ddr B_{s+\tau_y^{\infty,-}} -\gamma \int_0^tX_s\ddr s +\int_0^t\int_0^{X_{s-}}\int_{(0,1]}h\tilde{M}(\ddr s + \tau_y^{\infty,-},\ddr u,\ddr h) \\ +\int_0^t\int_0^{X_{s-}}\int_{(1,\infty]}h M(\ddr s+\tau_y^{\infty,-},\ddr u,\ddr h)-\int_0^tI(X_s)\ddr s 
\end{multline}
until the first exit time of $(0,\infty)$.
\end{proposition}

\begin{proof}
The shifted random measure
%\[
$M'(\ddr s,\ddr u,\ddr h) := M(\ddr s + \tau_y^{\infty,-}, \ddr u, \ddr h)$
%\]
is a Poisson random measure on \( \mathbb{R}_+ \times \mathbb{R}_+ \times \mathbb{R}_+ \) with the same intensity measure as \( M \), namely $\ddr s\ddr u \pi(\ddr h)$. Similarly, the shifted process
%\[
$\left(B'_s\right)_{s \geq 0} := \left(B_{s + \tau_y^{\infty,-}}\right)_{s \geq 0}$
%\]
is a Brownian motion with the same distribution as \( B \), see e.g. \cite[Page 16]{Berestycki}. Therefore, the stochastic equation \eqref{eds shift} falls into the same class as the stochastic equation \eqref{eds}, and we deduce the existence and uniqueness of \( \tilde{X}^y(\xi) \) from Proposition~\ref{Prop existence}.\end{proof}

\paragraph*{} For any positive random variable $\xi$, and $a,b \in [0,\infty]$, we set:
\begin{equation*}
\tilde{\tau}_a^{y,-}(\xi)=\inf\{t\geq 0, \ \tilde{X}_t^{y}(\xi)\leq a\} \quad , \quad \tilde{\tau}_b^{y,+}(\xi)=\inf\{t\geq 0, \ \tilde{X}_t^{y}(\xi)\geq b\}
\end{equation*}

\paragraph{} 
We now present two convergence lemmas for the family of processes \( \tilde{X}^y \).  
Lemma~\ref{lemma feller tilde} establishes a regularity property with respect to the initial condition, in the same spirit as Proposition~\ref{prop feller}, but for the shifted process \( \tilde{X}^y \).  
Lemma~\ref{lemma cv uniforme vers tilde} then provides a central convergence result.

\begin{lemma}
\label{lemma feller tilde}
Assume that $I$ satisfies the one-sided Lipschitz condition B3. Fix $y\in (a,\infty)$, $\xi_n, \xi$ positive random variables such that $\xi_n \rightarrow \xi$ a.s. when $n\in \mathbb{N}$ and $n \rightarrow \infty$.
Then, for any $t\geq 0$ :
\begin{equation}
\underset{n \uparrow \infty \atop n \in \mathbb{N}}{\lim} \ \underset{s\leq t}{\sup} \ \vert \tilde{X}_s^{y}(\xi_n)-\tilde{X}_s^{y}(\xi)\vert = 0 \ \ \textup{a.s.}.
\end{equation}
\end{lemma}

\begin{proof}
The proof uses the same arguments as the proof of Proposition \ref{prop feller}. 
\end{proof}

\begin{lemma}
\label{lemma cv uniforme vers tilde}
Fix $t \geq 0$ and $y\in(a,\infty)$.
\begin{equation}
\underset{k\rightarrow \infty\atop k\in \mathbb{N}}\lim \ \underset{s\leq t}{\sup} \ \left( \tilde{X}_s^{y}(y) - X^k_{s+\tau_y^{\infty,-}} \right) =0 \ \ \ \textup{a.s.}
\end{equation}

%when $k\rightarrow \infty$.
\end{lemma}

\begin{proof}
By the strong Markov property, we have the distributional identity
\[
\left(X^k_{\tau_y^{-}(k)+t}, \ t \geq 0\right) \overset{d}{=} \left(\tilde{X}_t^y(y), \ t \geq 0\right).
\]
Since \( \tilde{X}_t^y(y) \to y \) almost surely as \( t \to 0 \), and \( \tau_y^{k,-} \uparrow \tau_y^{\infty,-} \) almost surely as \( k \to \infty \), it follows that for any \( \epsilon > 0 \), there exist \( \delta > 0 \) and \( k_0 \in \mathbb{N} \) such that:
\begin{align*}
\mathbb{P}\left(\sup_{0 \leq t \leq \delta} \left| \tilde{X}_t^y(y) - y \right| > \epsilon\right) \leq \frac{\epsilon}{2},\qquad \forall k \geq k_0, \quad \mathbb{P}\left(\tau_y^{-}(\infty) > \tau_y^{-}(k) + \delta\right) \leq \frac{\epsilon}{2}.
\end{align*}

Then, for all \( k \geq k_0 \), we have:
\begin{align*}
\mathbb{P}\left(\left| X^k_{\tau_y^{\infty,-}} - y \right| > \epsilon \right) 
&\leq \mathbb{P}\left(\left| X^k_{\tau_y^{\infty,-}} - y \right| > \epsilon, \ \tau_y^{k,-} \leq \tau_y^{\infty,-} \leq \tau_y^{k,-} + \delta \right) + \frac{\epsilon}{2} \\
&\leq \mathbb{P}\left( \sup_{0 \leq t \leq \delta} \left| X^k_{\tau_y^{k,-} + t} - y \right| > \epsilon \right) + \frac{\epsilon}{2} \\
&= \mathbb{P}\left( \sup_{0 \leq t \leq \delta} \left| \tilde{X}_t^y(y) - y \right| > \epsilon \right) + \frac{\epsilon}{2} \\
&\leq \epsilon.
\end{align*}

This shows that \( X^k_{\tau_y^{\infty,-}} \to y \) in probability as \( k \to \infty \). Moreover, since the mapping \(k\mapsto X^k_{\tau_y^{\infty,-}}\) is almost surely non-decreasing, convergence actually holds almost surely:
\[
X^k_{\tau_y^{\infty,-}} \longrightarrow y \quad \text{a.s. as } k \to \infty.
\]

By uniqueness of solutions to the stochastic equation \eqref{eds shift}, we also have:
\[
X^k_{\tau_y^{\infty,-} + t} = \tilde{X}_t^y\left(X_{\tau_y^{\infty,-}}^{k}\right) \quad \text{a.s.}
\]
Hence, applying Lemma~\ref{lemma feller tilde}, we obtain:
\[
\sup_{s \leq t} \left| X^k_{\tau_y^{\infty,-}+s} - \tilde{X}_s^y(y) \right| 
= \sup_{s \leq t} \left| \tilde{X}_s^y\left(X^k_{\tau_y^{\infty,-}}\right) - \tilde{X}_s^y(y) \right| 
\underset{k \to \infty}{\longrightarrow} 0 \quad \text{a.s.}
\]
\end{proof}

\paragraph{}
Throughout the remainder of this section, the integers $n$ in the definition of $\tilde{X}^n(n)$ are assumed to be larger than $a$, so that $\tau_n^{\infty,-}<\infty$ a.s. By uniqueness of the almost sure limit and as a consequence of Lemma~\ref{lemma feller tilde}, we obtain:
\begin{equation}
\label{eq construction xinf}
\forall t \geq \tau_n^{\infty,-}, \quad X_t^{\infty} = \tilde{X}_{t - \tau_n^{\infty,-}}^{n}(n).
\end{equation}
Since \( \tau_n^{\infty,-} \downarrow 0 \) a.s. as \( n \to \infty \), this provides an explicit representation of \( X^{\infty}_t \) for $t\in (0, \infty) $. Moreover, for any \( n \geq m \) and \( t \geq \tau_n^{\infty,-} \geq \tau_m^{\infty,-} \), Lemma~\ref{lemma cv uniforme vers tilde} ensures that 
\[
\tilde{X}_{t - \tau_n^{\infty,-}}^{n}(n) = \tilde{X}_{t - \tau_m^{\infty,-}}^{m}(m) \quad \text{a.s.},
\]
so the expression in \eqref{eq construction xinf} is almost surely well-defined and independent of \( n \) for \( t > 0 \). Since the processes \( \tilde{X}^y \) have càdlàg paths, the process \( X^{\infty} \) also inherits this property. Finally, because each \( \tilde{X}^n(n) \) solves the stochastic differential equation \eqref{eds shift}, the representation \eqref{eq construction xinf} implies that \( X^{\infty} \) satisfies the same type of equation. This is formalized in the following lemma.

\begin{lemma}
\label{lemma verifie eds} The process $X^{\infty}$ satisfies $X^{\infty}_t\rightarrow \infty$ a.s. when $t \downarrow 0$ and is a strong solution of the stochastic equation :
\begin{multline}
\label{eds infini}
X_t=X_r+ \sigma\int_r^t\sqrt{X_s}\ddr B_s -\gamma\int_r^tX_s\ddr s +\int_r^t\int_0^{X_{s-}}\int_0^1h\tilde{M}(\ddr s,\ddr u,\ddr h) \\ +\int_r^t\int_0^{X_{s-}}\int_1^{\infty}h M(\ddr s,\ddr u,\ddr h)
-\int_r^tI(X_s)\ddr s, \quad 0< r\leq t.
\end{multline}
\end{lemma}

\begin{proof}
Fix \( t > 0 \). There exists \( n \in \mathbb{N} \) such that \( 0 \leq \tau_n^{\infty,-} \leq t \) almost surely. Then, by Lemma~\ref{lemma cv uniforme vers tilde},
\[
\forall k \in \mathbb{N}, \quad X^{\infty}_t = \tilde{X}_{t - \tau_n^{\infty,-}}^n(n) \geq X_t^k \quad \text{a.s.}
\]
Hence, for any \( x \in (0, \infty) \), we have \( X^{\infty}_t \geq X_t^x \). On the event \( \{0 < t < \tau_y^{x,-}\} \), this yields \( X^{\infty}_t > y \). Since \( \tau_y^{x,-} \uparrow \tau_y^{\infty,-} \) a.s. as \( x \to \infty \), we deduce that \( X^{\infty}_t > y \) on the event \( \{0 < t < \tau_y^{\infty,-}\} \). As \( \tau_y^{\infty,-} > 0 \) a.s., this proves that
\[
X^{\infty}_t \to \infty \text{ as } t \to 0 \text{ a.s.}
\]

Now fix \( t \geq r > 0 \), and let \( n \in \mathbb{N} \) such that \( \tau_n^{\infty,-}\leq r \) a.s. By definition of \( \tilde{X}^n(n) \), we have
\begin{align*}
\tilde{X}_t^n(n) &= n 
+ \sigma \int_0^t \sqrt{\tilde{X}_s^n(n)} \, \ddr B_{s + \tau_n^{\infty,-}} -\gamma\int_0^t\tilde{X}_s^n(n) \, \ddr s \\
&\quad + \int_0^t \int_0^{\tilde{X}_{s-}^n(n)} \int_0^1 h \, \tilde{M}(\ddr s + \tau_n^{\infty,-}, \ddr u, \ddr h) \\
&\quad + \int_0^t \int_0^{\tilde{X}_{s-}^n(n)} \int_1^\infty h \, M(\ddr s + \tau_n^{\infty,-}, \ddr u, \ddr h) \\
&\quad - \int_0^t I(\tilde{X}_s^n(n)) \, \ddr s.
\end{align*}

Making the time change \( s \mapsto s - \tau_n^{-}(\infty) \), we obtain
\begin{align*}
\tilde{X}_{t - \tau_n^{\infty,-}}^n(n) &= n 
+ \sigma \int_{\tau_n^{\infty,-}}^t \sqrt{\tilde{X}_{s - \tau_n^{\infty,-}}^n(n)} \, \ddr B_s -\gamma \int_{\tau_n^{\infty,-}}^t\tilde{X}_{s-\tau_n^{\infty,-}}^n(n) \, \ddr s\\
&\quad + \int_{\tau_n^{\infty,-}}^t \int_0^{\tilde{X}_{(s - \tau_n^{\infty,-}) -}^n(n)} \int_0^1 h \, \tilde{M}(\ddr s, \ddr u, \ddr h) \\
&\quad + \int_{\tau_n^{\infty,-}}^t \int_0^{\tilde{X}_{(s - \tau_n^{\infty,-}) -}^n(n)} \int_1^\infty h \, M(\ddr s, \ddr u, \ddr h) \\
&\quad - \int_{\tau_n^{\infty,-}}^t I\big(\tilde{X}_{s - \tau_n^{\infty,-}}^n(n)\big) \, \ddr s.
\end{align*}

Since \( X^{\infty}_s = \tilde{X}_{s - \tau_n^{\infty,-}}^n(n) \) for all \( s \geq \tau_n^{\infty,-} \), this simplifies to:
\begin{align*}
X^{\infty}_t &= n 
+ \sigma \int_{\tau_n^{\infty,-}}^t \sqrt{X^{\infty}_s} \, \ddr B_s +\gamma \int_{\tau_n^{\infty,-}}^tX^{\infty}_s \, \ddr s\\
&\quad + \int_{\tau_n^{\infty,-}}^t \int_0^{X_{s-}(\infty)} \int_0^1 h \, \tilde{M}(\ddr s, \ddr u, \ddr h) \\
&\quad + \int_{\tau_n^{\infty,-}}^t \int_0^{X_{s-}(\infty)} \int_1^\infty h \, M(\ddr s, \ddr u, \ddr h) \\
&\quad - \int_{\tau_n^{\infty,-}}^t I(X^{\infty}_s) \, \ddr s.
\end{align*}

A similar expression holds for \( X^{\infty}_r \), so subtracting yields:
\begin{align*}
X^{\infty}_t - X^{\infty}_r &= \sigma \int_r^t \sqrt{X^{\infty}_s} \, \ddr B_s - \gamma \int_r^t X^{\infty}_s \, \ddr s\\
&\quad + \int_r^t \int_0^{X_{s-}^{\infty}} \int_0^1 h \, \tilde{M}(\ddr s, \ddr u, \ddr h) \\
&\quad + \int_r^t \int_0^{X_{s-}^{\infty}} \int_1^\infty h \, M(\ddr s, \ddr u, \ddr h) \\
&\quad - \int_r^t I(X^{\infty}_s) \, \ddr s.
\end{align*}
\end{proof}

We have shown that the process starting from infinity $X^{\infty}$ is a solution to the stochastic differential equation \eqref{eds shift} with initial state $\infty$. The following lemma establishes the uniqueness of this solution, which follows naturally from the uniqueness result for the solution of \eqref{eds} stated in Proposition \ref{Prop existence} and the comparison property (CP1).

\begin{lemma}
\label{lemma unicite eds} There is a unique solution $X$ of \eqref{eds infini}, that comes down from infinity and such that $X_t \rightarrow \infty$ a.s. as $t \downarrow 0$.
\end{lemma}

\begin{proof}
Let us first recall an important fact : for any $0 < r$, the solution $\left(Z_t, \ t\geq r\right)$ to the stochastic differential equation
\begin{multline}
\label{eds cut}
Z_t = z + \sigma\int_r^t \sqrt{Z_{s-}} \, \ddr B_s - \gamma \int_r^t Z_s \, \ddr s + \int_r^t \int_0^{Z_{s-}} \int_0^1 h \, \tilde{M}(\ddr s,\ddr u,\ddr h) \\
+ \int_r^t \int_0^{Z_{s-}} \int_1^{\infty} h \, M(\ddr s,\ddr u,\ddr h) - \int_r^t I(Z_s) \, \ddr s
\end{multline}
satisfies a comparison principle with respect to the initial value $z$, similar to property (CP1).

Let $Y$ be a solution of \eqref{eds infini} such that $Y_t \to \infty$ a.s. as $t \to 0$ and $Y$ comes down from infinity, i.e. $$\forall t>0, \quad Y_t<\infty \qquad \textup{a.s.}.$$ Fix $x > 0$. 
By càdlàg property and Fatou's lemma, we have $\underset{t\rightarrow 0}{\liminf} \, \mathbb{P}\left(X_t^x\leq Y_t\right)=1$. By the Markov property at time $t$ and by comparison:
$$\mathbb{P}\left(X_{t+s}^x\leq Y_{t+s}, \ \forall s\geq 0, \ X_t^x\leq Y_t\right)=\mathbb{P}\left(X_t^x\leq Y_t\right).$$
By letting $t$ go to $0$, using right continuity, we get:
$$\mathbb{P}\left(X_s^x\leq Y_s, \ s\geq 0\right)=1.$$
Thus:
\begin{equation}
\label{ineq unic 1}
\forall s\geq 0, \quad X_s^{\infty}\leq Y_s, \quad \text{a.s.}.
\end{equation}

Now let $\left(X^x_{\epsilon,t}, \ t \geq \epsilon\right)$ be the pathwise unique solution of:
\begin{multline}
X_t = x + \sigma \int_{\epsilon}^t \sqrt{X_s} \, \ddr B_s - \gamma \int_\epsilon^t X_s \, \ddr s + \int_{\epsilon}^t \int_0^{X_{s-}} \int_0^1 h \, \tilde{M}(\ddr s,\ddr u,\ddr h) \\
+ \int_{\epsilon}^t \int_0^{X_{s-}} \int_1^{\infty} h \, M(\ddr s,\ddr u,\ddr h) - \int_{\epsilon}^t I(X_s) \, \ddr s
\end{multline}
By the comparison property, %the process is non-decreasing a.s. in $x$. Define 
the following almost sure limit is well-defined
\[
X^{\infty}_{1/n,t} := \lim_{x \to \infty} X^x_{1/n,t},\ \ \forall t \geq 1/n. 
\]
Conditionally on $Y_{1/n}$, since $Y_{1/n} < \infty$ a.s., by comparison again, one has 
% with probability $1$ :
% \begin{equation}
% \forall t \geq 1/n, \quad Y_t \leq X^{\infty}_{n,t}.
% \end{equation}
%This implies that :
\begin{equation}
\label{ineq unic 2}
\forall t \geq 1/n, \quad Y_t \leq X_{1/n,t}^{\infty} \quad \text{a.s.}.
\end{equation}
Next, we perform the time shift $s \mapsto s - 1/n$. The process $\left(X^x_{t-1/n}, \ t \geq 1/n\right)$ satisfies:
\begin{multline}
X_t = x + \sigma \int_{1/n}^t \sqrt{X_s} \, \ddr B_{s - 1/n} - \gamma \int_{1/n}^t X_s \, \ddr s \\
+ \int_{1/n}^t \int_0^{X_{s-}} \int_0^1 h \, \tilde{M}(\ddr s - 1/n, \ddr u, \ddr h)
+ \int_{1/n}^t \int_0^{X_{s-}} \int_1^{\infty} h \, M(\ddr s - 1/n, \ddr u, \ddr h) \\
\!- \int_{1/n}^t I(X_s) \, \ddr s
\end{multline}
%By the same arguments used 
The processes $\big(X^x_{1/n,t}, t \geq 1/n\big)$ and $\big(X^x_{t - 1/n}, t \geq 1/n\big)$ have the same law since $\left(B_{s - 1/n}, s \geq 1/n\right)$ and $\tilde{M}(\ddr s - 1/n, \ddr u, \ddr h)$ are independent of each other, and with the same laws as $(B_s,s\geq 0)$ and $M(\ddr s,\ddr u,\ddr h)$, see the proof of Proposition \ref{prop feller}. Combining inequalities \eqref{ineq unic 1} and \eqref{ineq unic 2}, and passing to the limit $x \to \infty$, $n \to \infty$, we obtain that the processes $\left(X^{\infty}_t, t > 0\right)$ and $\left(Y_t, t > 0\right)$ are equal in distribution. Since $\forall t>0, \ X^{\infty}_t\leq Y_t$ a.s., it follows that the two processes are indistinguishable.
\end{proof}

\paragraph{}
Recall the definition of the metric $\rho_{\infty}$ on the space $D$ of $[0, \infty]$-valued càdlàg functions \eqref{def rhoinf}.
It remains to prove the local uniform convergence in the Skorokhod space $D$ of the process started from $x$ towards the process started from infinity, as $x \to \infty$. In our first construction of the process starting from infinity, we established the almost sure pointwise convergence
\[
\forall t \geq 0, \quad X_t^x \xrightarrow[x \to \infty]{\text{a.s.}} X^{\infty}_t.
\]
The following lemma establishes the stronger convergence in $(D, \rho_{\infty})$.
%, which is local uniform convergence.

\begin{lemma}
\label{lemma cv rho infinie} For all $t\in [0,\infty)$, \begin{equation}\label{lem:unifconv}\left(X_s^x, \ s\leq t\right) \underset{x\rightarrow \infty}{\longrightarrow} \left(X^{\infty}_s, \ s\leq t\right) \qquad \textup{in} \quad  (D,\rho_{\infty}).\end{equation}
When furthermore the $\mathrm{CBDI}$ process $X^\infty$ gets extinct in finite time a.s. the convergence \eqref{lem:unifconv} holds for $t=\infty$.
\end{lemma}

\begin{proof}
Fix $t \geq 0$ and $n \in \mathbb{N}$. We know that 
%\[
$\tau_n^{k,-} \uparrow \tau_n^{\infty,-}\quad \text{a.s. as } k \to \infty.$
%\]
Since the process has no negative jumps, it follows that $\tau_n^{\infty,-} > \tau_{2n}^{\infty,-}$ almost surely. Therefore we can select $k \in \mathbb{N}$ such that 
$\tau_n^{k,-} > \tau_{2n}^{\infty,-}$ a.s.. Fix now $x \geq k$. Using the inequality $1 - e^{-z} \leq z$ for any $z \geq 0$, we can bound the $\rho_\infty$-distance by
\begin{align*}
\rho_\infty\left(\left(X_s^x\right)_{s \leq t}, \left(X^{\infty}_s\right)_{s \leq t}\right) 
&\leq \sup_{0 \leq s \leq \tau_n^{k,-}} \left| e^{-X_s^x} - e^{-X^{\infty}_s} \right| \vee \sup_{\tau_{2n}^{\infty,-} \leq s \leq t} \left| X^{\infty}_s - X_s^x \right|.
\end{align*}
We now control the second term using Lemma~\ref{lemma cv uniforme vers tilde}:
\begin{align*}
\sup_{\tau_{2n}^{\infty,-} \leq s \leq t} \left| X^{\infty}_s - X_s^x \right|
&= \sup_{0 \leq s \leq t - \tau_{2n}^{\infty,-}} \left| \tilde{X}_s^{2n}(2n) - X^x_{s + \tau_{2n}^{\infty,-}} \right| \\
&\leq \ \ \ \ \sup_{0 \leq s \leq t} \left| \tilde{X}_s^{2n}(2n) - X^x_{s + \tau_{2n}^{\infty,-}} \right| \xrightarrow[x \to \infty]{} 0.
\end{align*}
To control the first term, we use the fact that, by definition of $\tau_n^{k,-}$, we have
$\forall s \leq \tau_n^{k,-}, \ X_s(k) \geq n$ a.s.. By the comparison property (CP1), we also have
\[
\forall s \leq \tau_n^{k,-}, \quad X^{\infty}_s \geq X_s^x \geq X_s^k \geq n \quad \text{a.s.}.
\]
 Hence,
$\sup_{0 \leq s \leq \tau_n^{k,-}} \, \left| e^{-X_s^x} - e^{-X^{\infty}_s} \right| \leq 2e^{-n}.$
Since $n$ is arbitrary, by letting $x \to \infty$, we get the desired convergence
\[\rho_\infty\left(\left(X_s^x\right)_{s \leq t}, \left(X^{\infty}_s\right)_{s \leq t}\right) \xrightarrow[x \to \infty]{} 0.\]
Assume now that $X^{\infty}$ gets extinct almost surely. Since $\tau_0^{x,-}\leq \tau_0^{\infty,-}$ a.s., we have for all $x\in (0,\infty],$
$$\rho_{\infty}\big((X_t^x)_{t\geq 0},(X_t^{\infty})_{t\geq 0}\big)=\underset{0\leq t \leq \tau_0^{\infty,-}}{\sup} \, \Big| e^{-X_t^x}-e^{-X_t^{\infty}}\Big|\xrightarrow[x\rightarrow\infty]{}0 \qquad \textup{a.s..}$$
\end{proof}

\noindent\textbf{Proof of Theorem \ref{th descente eds}}.
The result is an immediate consequence of Lemmas \ref{lemma verifie eds}, \ref{lemma unicite eds} and \ref{lemma cv rho infinie}. \qed\\

\noindent \textbf{Acknowledgment}.
This work is part of my PhD thesis carried out under the supervision of Clément Foucart, to whom I am deeply grateful for his guidance and support. This work was done with the support of the European Union (ERC, SINGER, 101054787). Views and opinions
expressed are however those of the authors only and do not necessarily reflect those of the European
Union or the European Research Council. Neither the European Union nor the granting authority can
be held responsible for them.

\section{Appendix}
We establish below some technical or intermediary results that were needed in the proofs.
\subsection{Proof of Proposition \ref{lemma generator}: local martingale problem of CBDI}
%\begin{proof} 
The strong Markov property satisfied by the CBDI comes from the strong uniqueness of the solution of \eqref{eds}, see for instance \cite[proof of Theorem 1.1 p. 2949]{PLi}. Take $f \in D_\mathcal{X}\subset C^2$. Recall the expression of the generator \begin{align*}
    \mathcal{X}f(x):=-I(x)f'(x)&-\gamma xf'(x)+\frac{\sigma^2}{2}xf''(x) \nonumber \\
&+x\int_0^{\infty}\left(f(x+u)-f(x)-uf'(x)\mathbf{1}_{\{u\leq 1\}}\right)\pi(\ddr u), \ \ x\in D_f.
\end{align*}
We verify that the process
%\begin{center}
    $\left(f(X_t)-\int_0^t\mathcal{X}f(X_s)\ddr s\right)_{t\geq 0}$    
%\end{center}
is a local martingale by applying It\^o's formula with jumps, see \cite[Theorem 93 on page 59]{Situ}. For all $t\geq 0$,
\begin{align*}
f(X_t)&=f(X_0)-\int_0^tf'(X_s)(\gamma X_s+I(X_s))\ddr s+\sigma\int_0^tf'(X_s)\sqrt{X_s}\ddr B_s+\frac{\sigma^2}{2}\int_0^tf''(X_s)\ddr s \\
&\quad +\int_0^t\int_0^{\infty}\int_0^{\infty}\left[f\left(X_{s-}+h\mathbf{1}_{\{h>1, \, u\leq X_{s-}\}}\right)-f(X_{s-})\right]M(\ddr s,\ddr u,\ddr h) \\
&\quad + \int_0^t\int_0^{\infty}\int_0^{\infty}\left[f\left(X_{s-}+h\mathbf{1}_{\{h\leq 1, \, u\leq X_{s-}\}}\right)-f(X_{s-})\right]\tilde{M}(\ddr s,\ddr u,\ddr h)\\
&\quad + \int_0^t\int_0^{\infty}\int_0^{\infty}\left[f\left(X_{s}+h\mathbf{1}_{\{h\leq 1, \, u\leq X_s\}}\right)-f(X_s)-f'(X_s)h\mathbf{1}_{\{h\leq 1, \, u\leq X_s\}}\right]\ddr s  \ddr  u  \pi(\ddr h)\\
&= f(X_0)-\int_0^tf'(X_s)(\gamma X_s+I(X_s))\ddr s+\sigma\int_0^tf'(X_s)\sqrt{X_s}\ddr B_s+\frac{\sigma^2}{2}\int_0^tf''(X_s)\ddr s \\
&\quad +\int_0^t\int_0^{X_{s-}}\int_1^{\infty}\left[f\left(X_{s-}+h\right)-f(X_{s-})\right]M(\ddr s,\ddr u,\ddr h)\\ 
&\quad + \int_0^t\int_0^{X_{s-}}\int_0^1\left[f\left(X_{s-}+h\right)-f(X_{s-})\right]\tilde{M}(\ddr s,\ddr u,\ddr h) \\
&\quad + \int_0^t\int_0^{X_{s-}}\int_0^1\left[f\left(X_{s}+h\right)-f(X_s)-f'(X_s)h\right]\ddr s  \ddr  u  \pi(\ddr h) \\
&= f(X_0)-\int_0^tf'(X_s)(\gamma X_s+I(X_s))\ddr s+\sigma\int_0^tf'(X_s)\sqrt{X_s}\ddr B_s+\frac{\sigma^2}{2}\int_0^tf''(X_s)\ddr s \\
&\quad +\int_0^t\int_0^{X_{s-}}\int_1^{\infty}\left[f(X_{s-}+h)-f(X_{s-})\right]\tilde{M}(\ddr s,\ddr u,\ddr h) \\
&\quad + \int_0^t\int_0^{X_{s-}}\int_0^1\left[f(X_{s-}+h)-f(X_{s-})\right]\tilde{M}(\ddr s,\ddr u,\ddr h)\\
&\quad + \int_0^t\int_0^{\infty}\left[f(X_{s}+h)-f(X_{s})-hf'(X_{s})\mathbf{1}_{\{h\leq 1\}}\right]X_{s-}\ddr s\pi(\ddr h)\\
&=f(x)+\int_0^t\mathcal{X}f(X_s)\ddr s + \ \textup{local martingale}
\end{align*}
\qed
%\end{proof}
\subsection{Construction of the sequence $(\phi_k)$}
We adapt the construction given in the proof of Theorem IV-3.2 of \cite{IW}. Let $(a_n)_{n \in \mathbb{N}}$ be a decreasing sequence in $(0,\infty)$ defined by:
\[
a_0 = 1, \quad
a_n = \sqrt{\frac{1}{1 + \frac{n(n-1)}{2}}}, \quad \forall n \geq 1.
\]
Clearly, $a_n \to 0$ as $n \to \infty$. Observe that 
%\[
$\int_{a_{k+1}}^{a_k} \frac{2}{kz} \, \ddr z = 2, \ \forall k \in \mathbb{N}.$
%\]
This ensures the existence of a function $\psi_k \in C^2$ supported on $[a_{k+1}, a_k]$ such that:
\[
\int_{a_{k+1}}^{a_k} \psi_k(z) \, \ddr z = 1, \quad
\psi_k(z) \leq \frac{2}{kz} \quad \text{for all } z \in [a_{k+1}, a_k].
\]
Now define $\phi_k : \mathbb{R} \to \mathbb{R}$ by:
\[
\phi_k(z) =
\begin{cases}
0, & \text{if } z \leq 0, \\
\displaystyle\int_0^z \ddr y \int_0^y \psi_k(x) \, \ddr x, & \text{if } z > 0.
\end{cases}
\]
Then each $\phi_k$ is a smooth function in $C^2$ satisfying the desired properties : nonnegative, increasing on $(0, \infty)$, compactly supported in $(0, \infty)$, and constructed to localize mass within $(a_{k+1}, a_k)$ in a controlled way.
\bibliographystyle{amsalpha} 

\bibliography{biblio}

\end{document}